\documentclass{article}
\usepackage[a4paper, margin=1in]{geometry} 
\usepackage{graphicx} 
\usepackage{forest,tikz,amsfonts,amsmath,amssymb,amsthm}

\usepackage[a4paper, margin=1in]{geometry} 

\usepackage{subcaption} 

\usetikzlibrary{shapes.geometric, arrows}

\usepackage{ytableau}
\usepackage{multicol}

\usepackage{enumitem}

\usepackage{dsfont}

\usepackage{soul}

\usepackage{hyperref}
\hypersetup{
    colorlinks=true,
    linkcolor=blue,
    filecolor=magenta,      
    urlcolor=cyan,
    pdftitle={Overleaf Example},
    pdfpagemode=FullScreen,
    }

\usepackage{cleveref}

\def\V#1{{\mathbf #1}}

\def\Bern{\operatorname{Bern}}

\def\Unif{\operatorname{Unif}}

\def\Var{\operatorname{Var}}

\def\P{\mathbb{P}}
\def\E{\mathbb{E}}
\def\B{\operatorname{B}}

\def\U{\operatorname{U}}
\def\SO{\operatorname{SO}}

\def\Binom{\operatorname{Binom}}

\def\depth{\operatorname{depth}}

\def\lis{\operatorname{LIS}}
\def\lds{\operatorname{LDS}}

\newtheorem{lemma}{Lemma}
\newtheorem{proposition}{Proposition}
\newtheorem{theorem}{Theorem}
\newtheorem{corollary}{Corollary}
\newtheorem{conjecture}{Conjecture}
\newtheorem{remark}{Remark}

\allowdisplaybreaks

\title{Heights of butterfly trees}
\author{John Peca-Medlin\thanks{Department of Mathematics, University of California, San Diego, \href{mailto:jpecamedlin@ucsd.edu}{jpecamedlin@ucsd.edu}} \and Chenyang Zhong\thanks{Department of Statistics, Columbia University, \href{mailto:cz2755@columbia.edu}{cz2755@columbia.edu}}}
\date{}

\begin{document}

\maketitle

\begin{abstract}
    Binary search trees (BSTs) are fundamental data structures whose performance is largely governed by tree height. We introduce a block model for constructing BSTs by embedding internal BSTs into the nodes of an external BST—a structure motivated by parallel data architectures—corresponding to composite permutations formed via Kronecker or wreath products. Extending Devroye's result that the height $h_n$ of a random BST satisfies $h_n / \log n \to c^* \approx 4.311$, we show that block BSTs with $nm$ nodes and fixed external size $m$ satisfy $h_{n,m} / \log n \to c^* + h_m$ in distribution. We then study \textit{butterfly trees}: BSTs {with $N = 2^n$ nodes} generated from permutations built using iterated Kronecker or wreath products.  For simple butterfly trees (from iterated Kronecker products of $S_2$), we give a full distributional description showing polynomial height growth: $\E h_n^{\B} = \Theta(N^\alpha)$ with $\alpha = \log_2(3/2) \approx 0.58496$. For nonsimple butterfly trees (from wreath products), we prove power-law bounds: $cN^\alpha \cdot (1 + o(1)) \le \E h_n^{\B} \le dN^\beta\cdot (1 + o(1))$, with $\beta \approx 0.913189$.
    
\end{abstract}

\section{Introduction}\label{sec: intro}

A binary tree is a simple hierarchical data structure, where every node can only have two immediate children (i.e., a left or right child), starting with a single root node. A binary search tree (BST) is a binary tree where the labeling of the nodes reflects the inherent ordering structure, where each left child and right child has, respectively, smaller and larger label (i.e., key) than the parent node. Nodes with no children are called leaves, and the depth of a node in the BST is the length of the path connecting the root to the node. The height $h$ of the BST is then the maximal depth among all nodes within the tree. The height is a fundamental property of BSTs as it directly determines the time complexity of standard operations such as search, insertion, and deletion, which are performed in $O(h)$ time. Binary search trees (BSTs) are widely used for managing ordered data in applications such as priority queues, file systems, and database indexing~\cite{CLRS09}. Recent work has also explored BST variants optimized for parallelism and cache efficiency, particularly in the context of concurrent data structures and large-scale data processing~\cite{aksenov2017concurrency,sun2022joinable,Blelloch2016parallel}. These developments motivate the study of alternative BST constructions that align more naturally with recursive data layouts and hierarchical memory architectures---directions that we pursue in this work.

A BST with $n$ nodes can be generated from an array $\V x \in \mathbb R^n$ of distinct real numbers by iteratively applying the insertion operation using $\V x$ as the sequence of keys, starting with the empty tree. We will denote this tree  as $\mathcal T(\V x)$. 
The resulting underlying structure depends solely on the ordering of the keys, which can be derived from the corresponding permutation generated by the ordering.  Let $\depth_{\V x}(x_k)$ denote the depth of node $x_k$, which returns the number of edges traversed in $\mathcal T(\V x)$ connecting the root $x_1$ to node $x_k$, with then the tree height $h(\mathcal T(\V x)) = \max_{k} \depth_{\V x}(x_k)$. For example, the array $\V x = [0.65, 0.70, 0.27, 0.67, 0.16, 0.75]^T$ yields the permutation $\pi = 352416 \in S_6$, with the construction of $\mathcal T(\pi)$ outlined in \Cref{fig:352416}, that has height 2.


The behavior of BSTs generated from random permutations $\mathcal T(\pi)$ when $\pi \sim \Unif(S_n)$ has been extensively studied (see \cite{devroye1986note,devroye1990height,devroye1995variance,drmota2002variance,reed2003height,robson2002constant}). A landmark result is Devroye's theorem on the limiting height:
\begin{theorem}[\cite{devroye1986note}]\label{thm:devroye}
    If $\pi \sim \Unif(S_n)$ and $h_n = h(\mathcal T(\pi))$, then $h_n/\log n$ converges in probability and in $L^p$  for $p \ge 1$ to $c^* \approx 4.311$ as $n \to \infty$, where $c^*$ denotes the unique solution to $x \log(2e/x) = 1$ for $x \ge 2$.
\end{theorem}
{Prior to the specific study of random BSTs, earlier work established that the expected height of random rooted planar trees as well as (unlabeled) binary trees with $n$ nodes grows as $\Theta(\sqrt{n})$~\cite{debruijn1972height,flajolet1982average}. More recent work on BSTs has extended Devroye's result to broader classes of random permutations, including Mallows trees~\cite{addario2021height}, record-biased trees~\cite{corsini2023height} and permuton-sampled trees~\cite{corsini2025binary}.}

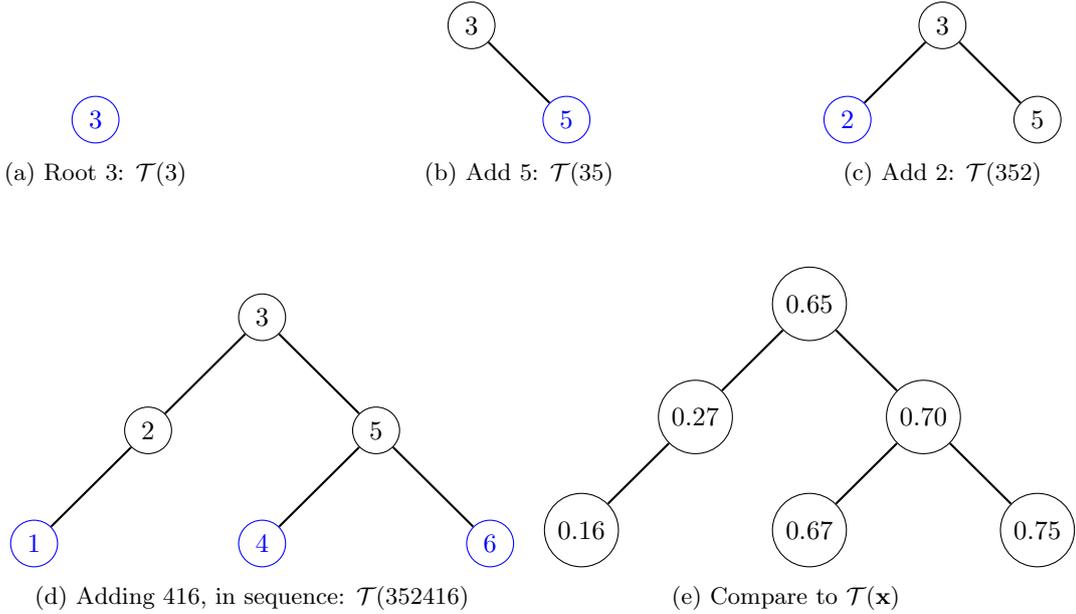
\begin{figure}
    \centering
    
    \begin{subfigure}{0.3\textwidth}
        \centering
        \begin{tikzpicture}[x=1.25cm, y=1.25cm]
            \node[circle, draw, blue] (n3a) at (0,0) {3};
        \end{tikzpicture}
        \caption{Root 3: $\mathcal T(3)$}
    \end{subfigure}
    \hfill
    \begin{subfigure}{0.3\textwidth}
        \centering
        \begin{tikzpicture}[x=1.25cm, y=1.25cm]
            \node[circle, draw] (n3b) at (0,0) {3};
            \node[circle, draw,blue] (n5b) at (1,-1) {5};
            \draw[thick] (n3b) -- (n5b);
        \end{tikzpicture}
        \caption{Add 5: $\mathcal T(35)$}
    \end{subfigure}
    \hfill
    \begin{subfigure}{0.3\textwidth}
        \centering
        \begin{tikzpicture}[x=1.25cm, y=1.25cm]
            \node[circle, draw] (n3c) at (0,0) {3};
            \node[circle, draw,blue] (n2c) at (-1,-1) {2};
            \node[circle, draw] (n5c) at (1,-1) {5};
            \draw[thick] (n3c) -- (n2c);
            \draw[thick] (n3c) -- (n5c);
        \end{tikzpicture}
        \caption{Add 2: $\mathcal T(352)$}
    \end{subfigure}



    \vspace{1cm} 

    \begin{subfigure}{0.4\textwidth}
        \centering
        \begin{tikzpicture}[x=1.5cm, y=1.5cm]
            \node[circle, draw] (n3d) at (0,0) {3};
            \node[circle, draw] (n2d) at (-1,-1) {2};
            \node[circle, draw] (n5d) at (1,-1) {5};
            \node[circle, draw,blue] (n1d) at (-2,-2) {1};
            \node[circle, draw,blue] (n4d) at (0,-2) {4};
            \node[circle, draw,blue] (n6d) at (2,-2) {6};
            \draw[thick] (n3d) -- (n2d);
            \draw[thick] (n3d) -- (n5d);
            \draw[thick] (n2d) -- (n1d);
            \draw[thick] (n5d) -- (n4d);
            \draw[thick] (n5d) -- (n6d);
        \end{tikzpicture}
        \caption{Adding 416, in sequence: $\mathcal T(352416)$}
    \end{subfigure}
    \hspace{1pc}
%
    %
    \begin{subfigure}{0.4\textwidth}
        \centering
        \begin{tikzpicture}[x=1.5cm, y=1.5cm]
            \node[circle, draw] (n3d) at (0,0) {0.65};
            \node[circle, draw] (n2d) at (-1,-1) {0.27};
            \node[circle, draw] (n5d) at (1,-1) {0.70};
            \node[circle, draw] (n1d) at (-2,-2) {0.16};
            \node[circle, draw] (n4d) at (0,-2) {0.67};
            \node[circle, draw] (n6d) at (2,-2) {0.75};
            \draw[thick] (n3d) -- (n2d);
            \draw[thick] (n3d) -- (n5d);
            \draw[thick] (n2d) -- (n1d);
            \draw[thick] (n5d) -- (n4d);
            \draw[thick] (n5d) -- (n6d);
        \end{tikzpicture}
        \caption{Compare to $\mathcal T(\V x)$}
    \end{subfigure}

    \caption{Step-by-step construction of $\mathcal T(352416)$ next to $\mathcal T(\V x)$}
    \label{fig:352416}
\end{figure}

We focus on block BSTs, a hierarchical structure in which internal BSTs are placed at each node of an external BST {(see \Cref{sec: block} for a formal definition of this block operation on BSTs)}. This structure naturally arises in parallel implementations of BSTs (e.g., \cite{Chen, HHNA22}).  We will consider the cases when the internal BSTs are identical as well as distinct (viz., they are independent and {identically distributed} (iid) in the random BST cases); these then correspond respectively to uniform permutations formed using Kronecker and wreath products (see \Cref{sec:notation} for definitions of each product), which we will denote $S_n \wr S_m$ and $S_m \otimes S_n$ (note the ordering of the $n,m$ terms differ, as on the wreath product the external structure is determined by the right action while in the Kronecker product the external structure is determined by the left action), that each comprise permutation subgroups $S_m \otimes S_n \subset S_n \wr S_m \subset S_{nm}$. {(See \Cref{sec:notation} for an explicit description for each group structure.)} The recursive construction differs from other studied recursive BST structures, such as seen in \cite{Pittel_1994}.

We are interested in how this construction compares to the uniform case on the BST with $nm$ nodes. A particular distinction when studying the heights of block BSTs is how the internal BSTs are  joined using the external BST. The heights are not additive across the internal BSTs, as the external BST edges join internal BSTs by traversing either the top-left or top-right edges (denoted $\ell_n$ and $r_n$ for a general BST with $n$ nodes) of the preceding internal BST before connecting to the root of the following internal BST. These top edges are determined by special paths in the permutation: the top-right edge follows the left-to-right maxima, while the top-left edge follows the left-to-right minima (see \Cref{sec: uniform} for more details). These paths are traced by scanning the permutation from left to right, updating the current maximum or minimum as needed. For example, in \Cref{fig:352416}, the subsequences 356 and 321 are the respective left-to-right maxima and minima paths seen within $\pi = 352416$ which manifest as the top-right and top-left edges of the corresponding BST $\mathcal T(\pi)$. (Note, for instance, that 4 does not appear within the left-to-right maxima path since 5 is encountered before 4 in $\pi$.)

For comparison, Kronecker and wreath products of permutations have recently been studied in terms of their impact on the length of the longest increasing subsequence (LIS). The LIS problem has a very rich history of its own (see \cite{Romik_2015} for a thorough historical overview of the LIS problem). Among this history, we highlight the Vershik-Kerov theorem that $\lis(\pi_n)/\sqrt n$ converges in probability to 2 when ${\pi_{n}}\sim \Unif(S_n)$. Very recently, Chatterjee and Diaconis \cite{chatterjee2024vershik} extended the Vershik-Kerov theorem to consider the wreath product of permutations; they established if both $m,n$ are growing sufficiently fast (viz., $n/\log^2 m \to \infty$), then $\lis(\pi)/\sqrt{nm}$ converges in probability to 4 when $\pi \sim \Unif(S_n \wr S_m)$. A simple observation is the fact {that} one wreath product is enough to increase the LIS from the uniform permutation case by a factor of 2 while maintaining the square root scaling. So one could then ask how do repeated wreath products impact the LIS, and can these change the scaling for the LIS? In~\cite{PZ24}, we resolve this question by showing that $n$-fold Kronecker or wreath products increase the power-law scaling of the LIS from $\sqrt{N}$ to at least $N^\alpha$, where $\alpha = \log_2(3/2) \approx 0.58496$ and $N = 2^n$ is the length of the permutation. Here, we initiate a {parallel investigation for the height of block BSTs.} 

Our first (motivational) contribution mimics the Chatterjee-Diaconis result when applied to the heights:
\begin{theorem}\label{thm:1wr}
    Let $\pi \sim \Unif(S_n \wr S_m)$ or $\pi \sim \Unif(S_m \otimes S_n)$, and $h_{n,m} = h(\mathcal T(\pi))$. If $m$ is fixed, then $h_{n,m}/\log (nm)$ converges in distribution to $c^* + h_m$ as $n \to \infty$.
\end{theorem}
Here, $c^* \approx 4.311$ is the exact same constant from \Cref{thm:devroye}, and $h_m$ denotes the (random) height of a BST with $m$ nodes.  When $m=1$, this recovers Devroye’s original result. 

This shows that even a single layer of block structure (via one wreath or Kronecker product) increases the limiting height, while preserving logarithmic scaling. We conjecture that when $n$ is fixed and $m\to\infty$, the properly rescaled height converges to $c^*H_n^{(1)}$, where $H_n^{(k)} = \sum_{j=1}^n 1/{j^k}$ denotes the $n^{th}$ generalized harmonic number. For the case when both $n$ and $m$ grow, we conjecture the height scales as $\log n\log m$, exceeding the $\log(nm)$ scaling of a uniform $S_{nm}$ permutation.

The remainder of the paper explores how repeated Kronecker or wreath products affect the height of random BSTs, as well as the interplay between BST height and other permutation statistics such as the LIS, number of cycles, and left-to-right extrema. As an example, \Cref{fig: big BST} shows the BST for a permutation in $S_{30}$, highlighting LIS and LDS (longest decreasing sequence) paths. Unlike height, which follows downward paths constrained by the binary tree structure, the LIS can ``jump'' across branches—leading to polynomial, rather than logarithmic, scaling behavior.

\begin{figure}
        \centering
    \resizebox{\textwidth}{!}{
    \begin{tikzpicture}
    
    \node[circle, draw] (n3) at (1,2) {3};
    \node[circle, draw] (n4) at (2,1) {4};
    \node[circle, draw] (n6) at (3,1) {6};
    \node[circle, draw] (n5) at (3,3) {5};
    \node[circle, draw] (n1) at (3,5) {1};
    \node[circle, draw] (n7) at (4,2) {7};
    \node[circle, draw] (n8) at (4,4) {8};
    \node[circle, draw] (n2) at (4,6) {2};
    \node[circle, draw] (n10) at (5,5) {10};
    \node[circle, draw] (n9) at (5,3) {9};
    \node[circle, draw] (n11) at (6,4) {11};
    
    \node[circle, draw] (n12) at (7,7) {12};

    \node[circle, draw] (n13) at (7,4) {13};
    \node[circle, draw] (n14) at (8,5) {14};
    \node[circle, draw] (n16) at (8,1) {16};
    \node[circle, draw] (n20) at (9,6) {20};
    \node[circle, draw] (n17) at (9,2) {17};
    \node[circle, draw] (n15) at (9,4) {15};
    \node[circle, draw] (n18) at (10,3) {18};
    \node[circle, draw] (n21) at (11,3) {21};
    \node[circle, draw] (n22) at (12,4) {22};
    \node[circle, draw] (n25) at (13,5) {25};
    \node[circle, draw] (n23) at (13,3) {23};
    \node[circle, draw] (n24) at (14,2) {24};
    \node[circle, draw] (n26) at (14,3) {26};
    \node[circle, draw] (n27) at (15,4) {27};
    \node[circle, draw] (n28) at (16,3) {28};
    \node[circle, draw] (n29) at (17,2) {29};
    \node[circle, draw] (n30) at (18,1) {30};
    
    \draw[thick] (n4) -- (n3);
    \draw[thick] (n3) -- (n5);
    \draw[thick] (n5) -- (n7);
    \draw[thick] (n6) -- (n7);
    \draw[thick] (n5) -- (n8);
    \draw[thick] (n8) -- (n9);
    \draw[thick] (n8) -- (n10);
    \draw[thick] (n10) -- (n11);
    \draw[thick] (n2) -- (n10);
    \draw[thick] (n2) -- (n1);
    \draw[thick] (n2) -- (n12);
    \draw[thick] (n12) -- (n20);
    \draw[thick] (n20) -- (n14);
    \draw[thick] (n13) -- (n14);
    \draw[thick] (n14) -- (n15);
    \draw[thick] (n15) -- (n18);
    \draw[thick] (n18) -- (n17);
    \draw[thick] (n17) -- (n16);
    \draw[thick] (n20) -- (n25);
    \draw[thick] (n25) -- (n22);
    \draw[thick] (n22) -- (n21);
    \draw[thick] (n22) -- (n23);
    \draw[thick] (n23) -- (n24);
    \draw[thick] (n25) -- (n27);
    \draw[thick] (n27) -- (n26);
    \draw[thick] (n27) -- (n28);
    \draw[thick] (n28) -- (n29);
    \draw[thick] (n29) -- (n30);

    \draw[line width=0.5mm,  red,dotted] (n12) -- (n14);
    \draw[line width=0.5mm,  red,dotted] (n14) -- (n15);
    \draw[line width=0.5mm,  red,dotted] (n15) -- (n18);
    \draw[line width=0.5mm,  red,dotted] (n18) -- (n22);
    \draw[line width=0.5mm,  red,dotted] (n22) -- (n26);
    \draw[line width=0.5mm,  red,dotted] (n26) -- (n28);
    \draw[line width=0.5mm,  red,dotted] (n28) -- (n29);
    \draw[line width=0.5mm,  red,dotted] (n29) -- (n30);

    \draw[line width=0.5mm,  green,dotted] (n25) -- (n22);
    \draw[line width=0.5mm,  green,dotted] (n21) -- (n22);
    \draw[line width=0.5mm,  green,dotted] (n21) -- (n17);
    \draw[line width=0.5mm,  green,dotted] (n17) -- (n11);
    \draw[green, line width=0.5mm, bend right,dotted] (n7) to (n11);
    \draw[line width=0.5mm,  green,dotted] (n6) -- (n7);
    \draw[line width=0.5mm,  green,dotted] (n6) -- (n3);
    
    \end{tikzpicture}
    }
    \caption{$\mathcal T(\pi)$ for $\pi \in S_{30}$, with $h(\mathcal T(\pi)) = 6$ and particular LIS (from nodes 12 to 30) and LDS (from nodes 25 to 3) dotted paths are shown over the BST in red and green, respectively.}
    \label{fig: big BST}
\end{figure}
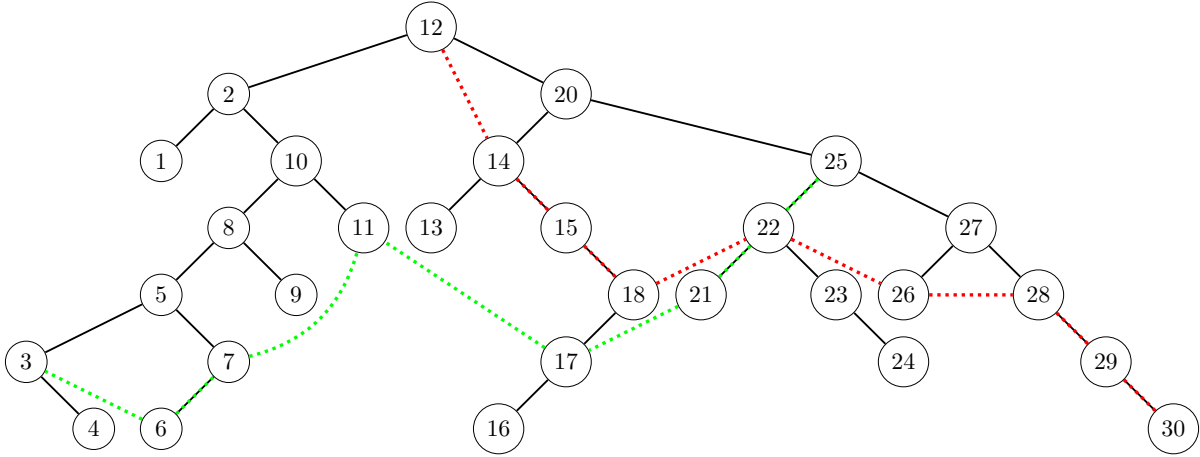

For this setting, we study the case of \textit{butterfly trees}. Butterfly trees are explicitly the BSTs generated using butterfly permutations, which were introduced in \cite{P24,PZ24}. We let $\mathcal T_{n,s}^{\B}$ and $\mathcal T_n^{\B}$ denote, respectively, the simple and nonsimple butterfly trees with $N = 2^n$ nodes. Butterfly permutations are permutations of length \( N = 2^n \), defined recursively via two constructions: \emph{simple} butterfly permutations \( \B_{n,s} \), built from \( n \)-fold Kronecker products (yielding the trees \( \mathcal{T}_{n,s}^{\B} \)), and \emph{nonsimple} butterfly permutations \( \B_n \), formed from \( n \)-fold wreath products (producing \( \mathcal{T}_n^{\B} \)). Both constructions begin from \(S_2 \). These permutations arise naturally as the pivoting patterns generated by Gaussian elimination with partial pivoting (GEPP) applied to random butterfly matrices—hence the name (see \Cref{sec: butterfly perms} for precise definitions). Structurally, the nonsimple butterfly permutations form a {2-}Sylow subgroup of \( S_{2^n} \), while the simple butterfly permutations form an abelian normal subgroup within them. {Moreover, this construction yields unique underlying binary tree shapes---so the study of heights of butterfly trees connects directly to {both} the study of heights of BSTs and binary tree shapes.}

In addition to their link to numerical linear algebra, butterfly permutations and matrices have deeper connections to the study of automorphism groups of rooted trees, with implications across multiple areas including quantum information theory, random graph theory, and error-correcting codes (see \Cref{sec: butterfly perms} for more details).

Our main contributions are in an average-case analysis of the heights of butterfly trees with $N = 2^n$ nodes, where we establish polynomial average height growth  rather than the logarithmic growth for uniform permutations. For the case of simple butterfly trees, we can provide  full nonasymptotic and asymptotic distributional descriptions of the heights:
\begin{theorem}\label{thm: sbt}
    Let $\pi \sim \Unif(\B_{n,s})$ and $h_n^{\B} = h(\mathcal T(\pi))$. Then $h_n^{\B} \sim 2^{X_n} + 2^{n - X_n} - 2$, where $X_n \sim \Binom(n,1/2)$, while $(\log_2 h_n^{\B} - n/2)/(\sqrt n/2)$ converges in distribution to $|Z|$ as $n \to \infty$, where $Z \sim N(0,1)$. In particular, $\E h_{n}^{\B} = 2(3/2)^n - 2 = c N^\alpha \cdot(1 + o(1))$ with $c=2$ and $\alpha = \log_2(3/2) \approx 0.58496$.
\end{theorem}
{We note in particular that not only does this expected height exceed the logarithmic heights of random BSTs, but also the $\Theta(\sqrt N)$ heights for uniform random rooted trees.}

The main ideas needed for the proof of \Cref{thm: sbt} follow from establishing $h_n^{\B} = \ell_n^{\B} + r_n^{\B}$, along with then showing $\ell_n^{\B} = {\lds}(\pi) - 1$ and $r_n^{\B} = {\lis}(\pi)-1$ for random simple butterfly permutations $\pi \sim \Unif(\B_{n,s})$, which is combined with the results $\log_2 \lis(\pi) \sim \Binom(n,1/2)$ and $\lis(\pi)\cdot \lds(\pi) = N$ that were each established in \cite{PZ24}. This marks a clear distinction from the uniform permutation of length $n$ case: $h_n$, $\ell_n$, and $r_n$ converge in probability, respectively, to $c^*\approx 4.311,1,1$ when scaled by $\log n$, so typically $h_n > 2(\ell_n + r_n)$, while also $\lis(\pi)$ is of a polynomial order $N^{1/2}$, much larger than the logarithmic scalings observed for the other statistics. For butterfly trees, the heights and LIS on average are now on the exact same order. We also highlight additional connections of the heights to degree distributions of vertices on the comparability graph for the Boolean lattice with $n$-element ground set, as well as derive a full distributional description for the maximal width of a uniform simple butterfly tree.

Our next contribution is in regard to $h_n^{\B}$ for the case of nonsimple butterfly trees. In this setting, we no longer can provide full distributional descriptions. Instead, we focus on establishing lower and upper power-law bounds on the first moment $\E h_n^{\B}$.
\begin{theorem}\label{thm: nsbt}
    Let $\pi \sim \Unif(\B_n)$ and $h_n^{\B} = h(\mathcal T(\pi))$. Then $cN^\alpha \cdot (1 + o(1)) \le \E h_n^{\B} \le d N^\beta \cdot (1 + o(1))$ for $\alpha \mathrel{{=}}{\log_2(3/2)} \approx 0.58496$, $\beta \mathrel{{=}}{\log_2} \xi \approx 0.913189$ {for $\xi = (1 + \sqrt 2 + \sqrt{2\sqrt 2 -1})/2$}, $c = 2$, and $d\mathrel{{=}} {2/(\sqrt 2 + \sqrt{2\sqrt 2 - 1} - 2)}\approx 2.60958$.
\end{theorem}
The lower bound matches exactly the first moment for the simple case in \Cref{thm: sbt}. A similar behavior was observed for the $\lis$ of nonsimple butterfly permutations in \cite{PZ24}, where again only power-law bounds were attained for the first moment $\E \lis(\pi)$. The main tools for \Cref{thm: nsbt} involve leveraging and iterating recursive relationships between the first two moments of $h_n^{\B}$, along with establishing additional properties for the top-left and top-right edge lengths---$\ell_n^{\B}$ and $r_n^{\B}$---for butterfly trees. Unlike in the simple case, where $r_n^{\B}$ aligned precisely with the $\lis$ of the corresponding butterfly permutations, we instead show that $r_n^{\B}$ now aligns with the number of cycles for this permutation. The number of cycles of butterfly permutations was more thoroughly studied in \cite{PZ24} than the LIS question in the nonsimple setting, as full nonasymptotic and asymptotic results were attainable for this statistic (that have behavior determined by moments that satisfy further recursive relations). It is also less clear how $h_n^{\B}$ compares to $\ell_n^{\B} + r_n^{\B}$, which now can both exceed or be smaller than the height and has average value match the lower bound in \Cref{thm: nsbt}. 

{To complement the theoretical bounds in \Cref{thm: nsbt}, we provide empirical evidence supporting the sublinear power-law scaling of the expected height of nonsimple butterfly trees. While the lower bound becomes strict from $n = 4$ onward by exact enumeration, larger-scale sampling suggests the lower bound may actually be asymptotically tight.
} 


\paragraph{Future directions and extensions.} 
A natural next step is to refine our understanding of the nonsimple butterfly case by seeking an exact expression for the average height of random butterfly trees and (further) a full distributional description. 
Additionally, establishing sharper asymptotic results for general block BSTs with non-fixed external block size presents another important direction for future work, as well as exploring similar directions for $p$-nary butterfly trees.

Beyond uniform, it would be valuable to explore butterfly and block permutation models with additional structural biases. One promising direction is to introduce Mallows-type distributions over butterfly or Kronecker-block permutations, which incorporate a bias parameter $q > 0$ to favor or penalize certain block structures. This would generalize our current results in the same way that the Mallows model generalizes the uniform case for permutations. It would also further extend classical results such as those of Devroye (cf. \Cref{thm:devroye}), whose analysis of the expected height of uniform BSTs is recovered in the 1-block case of our framework or the $q = 1$ balanced Mallows model. These extensions could lead to a broader theory of BST height behavior under structured and biased permutation models.

While the polynomial height of butterfly trees affects sequential performance (e.g., insertion or deletion operations), their recursive, compositional structure makes them highly amenable to parallel computation, suggesting future study of structural parameters relevant to parallel evaluation (e.g., contraction depth or Strahler-like complexity). Such investigations may further connect  to applications in parallel symbolic computation, expression evaluation, functional programming runtimes, and memory-efficient layout strategies for hierarchical architectures.

\paragraph{Outline.} The remaining paper is organized as follows. The remainder of \Cref{sec: intro} reviews relevant background on uniform permutations and introduces the formal definitions of butterfly permutations. The subsequent sections focus on the proofs of our main results and related discussion. \Cref{sec: block} addresses \Cref{thm:1wr} regarding block BSTs, while \Cref{sec: butterfly trees} covers both \Cref{thm: sbt,thm: nsbt} for butterfly trees along with a short discussion including empirical results.

\subsection{Preliminaries and notation}\label{sec:notation}

Let $\V e_i \in \mathbb R^n$ denote the standard basis vectors, with 1 in the $i^{th}$ component and 0 elsewhere. For matrices $A \in \mathbb R^{n \times n}$ and $B \in \mathbb R^{m \times m}$, let $A \otimes B \in \mathbb R^{nm \times nm}$ denote the Kronecker product formed by inserting $B$ into each element of $A$, as
$$
A \otimes B = \begin{bmatrix}
    A_{11}B & \cdots & A_{1n} B \\
    \vdots & \ddots & \vdots\\
    A_{n1}B & \cdots & A_{nn}B
\end{bmatrix}.
$$
The Kronecker product satisfies the mixed-product property, $(A \otimes B)(C \otimes D) = (AC)\otimes (BD)$, that can be read as the product of Kronecker products is the Kronecker product of products (assuming compatible matrix dimensions). We will also write $A \oplus B$ for the block diagonal matrix with diagonal blocks {$A,B$}. Let $S_n$ denote the symmetric group that consists of the permutations $\pi$ of a set of $n$ objects, and let $\mathcal P_n = \{P_\pi: \pi \in S_n\}$ denote the permutation matrices defined by $P_\pi \V e_i = \V e_{\pi(i)}$. For permutations $\pi \in S_n, \rho \in S_m$, let $\pi \otimes \rho \in S_n \otimes S_m \subset S_{nm}$ denote the corresponding permutation attained by $P_{\pi \otimes \rho} = P_\pi \otimes P_{\rho}$. Similarly, let $\pi = \pi_1 \oplus \pi_2$  denote the corresponding permutation with permutation matrices $P_\pi = P_{\pi_1} \oplus P_{\pi_2}$, while $\pi = \pi_1 \ominus \pi_2$ similarly denotes permutations where $P_{\pi} = P_{\pi_1} \ominus P_{\pi_2}$, where $$
A \ominus B = \begin{bmatrix}
    \V 0 & B \\ A & \V 0
\end{bmatrix}.
$$For subgroups $G \subset S_n$, $H \subset S_m$, let $G \wr H = G^m \rtimes H$ denote the wreath product formed by the action of $\rho \in H$ on the indices of an $m$-vector with components taken from copies of $G$ via $$\rho \cdot (x_1,x_2,\ldots,x_m) = (x_{\rho^{-1}(1)},x_{\rho^{-1}(2)},\ldots,x_{\rho^{-1}(m)}).$$ {To formalize the wreath product structure used in this paper, we realize $S_n \wr S_m$ as follows. An element is written as $\boldsymbol{\pi} \rtimes \rho \in S_{nm}$ where $\boldsymbol{\pi} = \bigoplus_{j=1}^m \pi_j \in S_{nm}$ with $\pi_j \in S_n$ and $\rho \in S_m$, whose associated permutation matrix is defined  to be $$P_{\boldsymbol{\pi} \rtimes \rho} = \Bigg(\bigoplus_{j=1}^m P_{\pi_j}\Bigg) \cdot (P_\rho \otimes \V I_n).$$ From this realization, the inclusion $S_m \otimes S_n \subset S_n \wr S_m$ is immediate:  if $\boldsymbol{\pi} = \bigoplus_{j=1}^m \pi = \pi^{\oplus m}$, then $\bigoplus_{j=1}^m P_\pi = \V I_m \otimes P_\pi$; the mixed-product property then yields $P_{\rho \otimes \pi} = P_{\boldsymbol{\pi} \rtimes \rho}$ so that $\rho \otimes \pi = \boldsymbol{\pi} \rtimes \rho$.}


For a sequence of random variables $X_1,X_2,\ldots$ and a random variable $Y$ (defined on the same sample space), we write $X_n \to_p Y$ to denote the sequence $X_n$ converges in probability to $Y$, $X_n \to_d Y$ to denote the sequence $X_n$ converges in distribution to $Y$, and $X_n \to_{L^p} Y$ to denote the sequence $X_n$ converges in $L^p$ to $Y$. We also recall standard results that will be utilized several times throughout, such as the equivalence of convergence in probability and distribution when the limit is constant, as well as Slutsky's theorem:
\begin{theorem}[Slutsky's theorem, \cite{slutsky1925stochastische}]\label{thm: slutsky}
    If $X_n,Y_n$ are sequences of random variables such that $X_n \to_d X$ and $Y_n \to_p \nu$ for a random variable $X$ and a constant $\nu$, then $X_n + Y_n \to_d X + \nu$ and $X_nY_n \to_d \nu X$.
\end{theorem}
Additionally, we will make use of the Subgroup Algorithm due to Diaconis and Shahshahani \cite{DiSh87}:
\begin{theorem}[Subgroup algorithm, \cite{DiSh87}]\label{thm: subgroup}
  Suppose that $G$ is a compact group, $H$ is a closed subgroup of $G$, and $G/H$ is the set of left-cosets of $H$ in $G$. If $x \sim \Unif(G/H)$ and $y \sim \Unif(H)$ are independent, then $xy \sim \Unif(G)$.  
\end{theorem}
\noindent This result allows for uniform sampling from $G$ by independently sampling a uniform coset representative and a uniform element from the subgroup $H$.


\subsection{Random BSTs}\label{sec: uniform}

Here we review briefly the background for the heights of random BSTs along with the lengths of the top-left and -right edges of the BST. Random BSTs are generated using uniform permutations, and have been the primary setting for the initial analysis of BST data structures. The height, $h_n$, for a BST with $n$ nodes is a fundamental statistic that governs the performance of common operations on the BST, such as search, insertion, and deletion, that each run in $O(h_n)$ time. A trivial lower bound of $h_n \ge \lfloor \log_2 n \rfloor$ is attained by the completely balanced BST (where each parent node has exactly two children when $n = 2^{m+1}-1$), while a worst-case of $h_n \le n-1$ is attained by the identity permutation (that consists only of a single right-edge path connecting the root 1 to the leaf node $n$). Further balancing mechanisms (e.g., AVL or Red-Black trees) can enforce $O(\log n)$ height but adds further complexity to the implementation that proves to be unnecessary on average (see below).

The average-case behavior of BSTs constructed under random insertion of $n$ distinct elements reveals insights into the practical performance of BSTs, as randomness tends to produce efficient tree structures. Devroye's seminal result (see \Cref{thm:devroye}) establishes that for a BST built from $n$ random insertions, the height is $O(\log n)$. Later work further establishes the variance is bounded \cite{devroye1995variance,drmota2002variance}. Hence, the average-case efficiency of BSTs aligns closely with the best-case $O(\log n)$ performance, differing only by a constant factor. Specifically, the height constant from the best-case balanced model is 
$1/\log 2 \approx 1.4427$, compared to the exact constant $c^* \approx 4.31107$ for the average-case model confirmed by Devroye.  Devroye's result built on previous work of Pittel \cite{pittel1984growing}, who established $h_n/\log n$ converges almost surely to a positive constant between 3.58 and 4.32. To recover the explicit constant $c^*$ requires heavy machinery under the hood, using a correspondence between random BSTs and branching processes along with the Hammersley-Kingman-Biggins theorem \cite{biggins1976first,hammersley1974postulates,kingman1975first} that provides a law of large numbers result for the minimum of a branching random walk (necessary to match the lower bound to $c^*$). 

Heights of more general random BST models have also been explored more recently, including the Mallows trees \cite{addario2021height}, record-biased trees \cite{corsini2023height}, and permuton sampled trees \cite{corsini2025binary}. Of these, we further highlight Mallows trees, which are generated using Mallows permutations \cite{mallows1957non}, that bias permutations closer or farther from the identity permutation through a scaling $q^{\operatorname{inv}(\pi)}$, where $\operatorname{inv}(\pi)$ is the number of inversions of the permutation $\pi$ with a weight $q \in (0,\infty)$; where $q = 1$ aligns with the uniform permutation setting (see also \cite{Basu_Bhatnagar_2017,bhatnagar2015lengths,gladkich2018cycle,He_2022,he2023cycles,Mueller_Starr_2013,zhong2023cycle,Zhong_2023} for further studies on Mallows models). When $q \in (0,1)$, then Mallows permutations are biased more toward the identity permutation, and so the limiting behavior relies more heavily on the height of the right subtree in the permutation. Hence, in particular, these rely also on the length of the top-right edge of the BST, connecting the root to the node $n$. This corresponds to the \textit{left-to-right maxima} path within the permutation, as we will review now. (Note that this applies exactly analogously to the left-to-right minima path, which corresponds to the top-left edge $\ell_n$ of the corresponding BST.)

{The left-to-right maxima path of a permutation $\pi \in S_n$ is obtained by recording, for each $k = 1,2,\ldots,n$, the maximum of the initial segment $\pi(1)\pi(2)\cdots \pi(k)$ (with repeated maxima removed).} This path stops as soon as $n$ is encountered, since $n$ is necessarily the overall maximum. This is seen in \Cref{fig: wreath BST example} as the permutation 216534 has {initial segments} 2, 21, 216 until the maximal 6 entry is encountered, so the left-to-right maxima path consists of just 26. This subsequence explicitly materializes in the corresponding BST for $\pi \in S_n$ as this is determined by the nodes visited in the path connecting the root $\pi(1)$ to $n$ (e.g., in \Cref{fig: wreath BST example}, the top-right edge consists of the root 2 and 6). Hence, if $r_n$ is the length of the top-right edge of the corresponding BST $\mathcal T(\pi)$, then the length of the left-to-right maxima path inside $\pi$ is $r_n + 1$.

For a uniform permutation, the distribution of the length of the left-to-right maxima has been understood since the 1960s, due to early work by R\'enyi \cite{renyi62}. In particular, one can show {that} $r_n + 1 \sim \Upsilon_n$, where $\Upsilon_n$ denotes the Stirling-1 distribution with probability mass function (pmf) given by
\begin{equation}\label{def: Stirling1 pmf}
    \P(\Upsilon_n = k) = \frac{|s(n,k)|}{n!}
\end{equation}
for $k = 1,2,\ldots,n$, where $s(n,k)$ denotes the Stirling numbers of the first kind. More traditionally, $\Upsilon_n$ models the \textit{number of cycles} $C(\pi)$ in the disjoint cycle decomposition of a uniform permutation $\pi$; this follows as $|s(n,k)|$ can be defined to provide the number of permutations in $S_n$ with $k$ cycles in their disjoint cycle decomposition. (It follows directly that \eqref{def: Stirling1 pmf} defines a \textit{bona fide} pmf, since $\sum_{k = 1}^n |s(n,k)| = |S_n| = n!$.) The Stirling numbers of the first kind can be computed explicitly using the  recursion
\begin{equation}\label{eq: stirling rec}
    |s(n+1,k)| = n |s(n,k)| + |s(n,k-1)|.
\end{equation}
A standard interpretation for this recursion is building up a permutation from $\pi \in S_n$ to a permutation of $S_{n+1}$ by considering where $n+1$ is placed between the elements from $\pi$, with the cases being split between $n+1$ being appended to the end of $\pi$ versus being placed at any earlier space. To verify $r_n + 1 \sim \Upsilon_n$, it suffices to show the left-to-right maxima can similarly decompose the permutations of length $n$ into sets of size $|s(n,k)|$ where the permutation has left-to-right maxima path of length $k$, and so also satisfy \eqref{eq: stirling rec}. This follows exactly as in the cycle case, but now following how 1 can be distributed within the permutation of length $n+1$: the placement of 1 only impacts the length of the left-to-right maxima path if $\pi(1) = 1$, as any other placement means 1 will never be visited on the maxima path. Standard results establish $\E \Upsilon_n = H_n^{(1)}$ and $\operatorname{Var}\Upsilon_n = H_n^{(1)} - H_n^{(2)}$. Since $H_n^{(1)} = \log n + \gamma + o(1)$ for $\gamma \approx 0.577$ the Euler-Mascheroni constant and $H_n^{(2)} = \frac{\pi^2}6 + o(1)$, along with the fact $h_n$ dominates $r_n + 1$, standard probability tools (e.g., Chebyshev's inequality, uniform integrability) yield: 
\begin{corollary}\label{c: Stirling limit}
    $\Upsilon_n/\log n$ converges in probability and in $L^p$ to 1 for all $p \ge 1$.
\end{corollary}
The left-to-right maxima and minima paths play a more pronounced role in the study of the heights of block BSTs, as will be explored in \Cref{sec: block}.

\subsection{Butterfly permutations}\label{sec: butterfly perms}

A large focus of our paper will be on butterfly trees, which are BSTs derived from two  classes of permutations of length $N = 2^n$ called \textit{butterfly permutations}. These are:
\begin{itemize}
    \item the \textit{simple butterfly permutations}, $\B_{n,s} = S_2 \otimes \cdots \otimes S_2 = S_2^{\otimes n}$, an $n$-fold iterated Kronecker product of $S_2$, and \\
\item the \textit{nonsimple butterfly permutations}, $\B_n = S_2 \wr  \cdots \wr S_2 = S_2^{\wr n}$, an $n$-fold iterated wreath product of $S_2$.
\end{itemize}
Both comprise subgroups of $S_{N}$, with $\B_n$ a 2-Sylow subgroup of $S_N$ while $\B_{n,s}$ an abelian normal subgroup of $\B_n$. The term \textit{butterfly permutations} originates from \cite{P24,PZ24}, where they arise as the permutation matrix factors resulting from applying Gaussian elimination with partial pivoting (GEPP) to scalar butterfly matrices---specifically, when GEPP is applied to uniformly sampled butterfly matrices, the resulting permutations are  uniformly sampled from $\B_n$ and $\B_{n,s}$. 

The associated butterfly matrices are recursively defined block-structured matrices generated using classical Lie group elements, especially $\SO(2)$ (or more generally $\U(2)$; see \cite{phd}). In particular, the simple butterfly matrices are of the form $\bigotimes_{j=1}^n \SO(2)$, which itself comprises an abelian subgroup of $\SO(N)$, and so has an induced Haar measure. Nonsimple butterfly matrices are formed recursively, starting with $\SO(2)$, and the order $N = 2^n$ butterfly matrices are formed by taking the product
$$
    \begin{bmatrix}
        C & S \\ -S & C
    \end{bmatrix}\begin{bmatrix}
        A_1 & \V 0 \\ \V 0 & A_2
    \end{bmatrix} = \begin{bmatrix}
        C A_1 & S A_2\\-S A_1 & C A_2
    \end{bmatrix}
$$
where $A_1,A_2$ are order $N/2$ butterfly matrices, and $C,S$ are symmetric commuting matrices that satisfy the Pythagorean matrix identity $S^2 + C^2 = \V I$. The scalar butterfly matrices assume the $(C,S)$ pair used at each recursive step are the scalar matrices $(C,S) = (\cos \theta,\sin\theta) \V I$ for some angle $\theta$; in this case, the above form condenses to 
$$
(R_\theta \otimes \V I)(A_1 \oplus A_2)
$$
for $R_\theta$ the (clockwise) rotation matrix of angle $\theta$. The simple butterfly matrices enforce the additional constraint that $A_1 = A_2$ at each intermediate recursive step, that allows the further collapsing of the above form into an iterated Kronecker product. The mixed-product property then derives the necessary Kronecker product form of the permutation matrix factors of simple scalar butterfly matrices. The nonsimple case is less direct (see \cite[Theorem 5]{PZ24}); for instance, the nonsimple butterfly matrices are not closed under multiplication but the associated GEPP permutation matrix factors still comprise a group.

The sizes of each butterfly permutation group are $$|B_{n,s}| = 2^n = N, \qquad |B_n| = 2^{2^n-1} = 2^{N-1}.$$ 
This confirms $\B_n$ as a 2-Sylow subgroup of $S_{2^n}$, as it matches the number of 2-factors in $|S_{2^n}| = (2^n)!$. Similar constructions generalize to $p$-Sylow subgroups of $S_{p^n}$ for other primes $p$ (see \cite{PZ24}). 

Butterfly matrices have been widely studied in numerical linear algebra, especially for their efficiency and structure in fast solvers \cite{lindquist2024generalizing,P24_gecp,P24,PT23,Tr19}. Their recursive structure,  when constructed from unitary groups, makes butterfly matrices natural models for quantum entanglement \cite{nielsen2010quantum}. This perspective extends to the associated butterfly permutations, whose group structure, low circuit complexity, and recursive symmetry make them well-suited as efficient data structures for quantum information processing.

These butterfly permutation groups also have deep connections to the automorphism group of the infinite rooted $p$-nary tree. For instance, Ab\'ert and Vir\'ag \cite{Abert_Virag_2005} resolve a classical question of Turán regarding the typical order of an element from such a $p$-Sylow subgroup utilizing this correspondence. This perspective connects butterfly groups to key questions in random graph theory, such as the girth of random Cayley graphs \cite{Gamburd_Hoory_Shahshahani_Shalev_Virag_2009} and optimal mixing in Markov chains \cite{boyd2009fastest}. In this paper, we focus on the binary case $p=2$.


\section{Heights of block BSTs}\label{sec: block}

We first study the height behavior of block BSTs---hierarchical structures formed by replacing each node of an external BST with an internal BST. These arise naturally in settings involving the parallelization of BST operations. We now focus on two combinatorially natural ways of generating these structures: via the Kronecker product and the wreath product of permutations. {Formally, we define the simple block BST composing external BST $\mathcal T(\rho)$ with $\rho \in S_m$ with internal BST $\mathcal T(\pi)$ with $\pi \in S_n$ as 
$$
\mathcal T(\rho) \, \boxdot \, \mathcal T(\pi) := \mathcal T(\rho \otimes \pi),
$$ 
where the Kronecker product $\rho \otimes \pi \in S_{nm}$ is the induced permutation from the block permutation matrix $P_{\rho \otimes \pi} = P_\rho \otimes P_\pi$. We can further define the block product of external BST $\mathcal T(\rho)$ with multiple internal BSTs $\mathcal T(\pi_j)$ with $\pi_j \in S_n$ for $j = 1,2,\ldots,m$,  as 
$$
\mathcal T(\rho) \, \boxdot  \, (\mathcal T(\pi_1),\ldots,\mathcal T(\pi_m)) := \mathcal T(\boldsymbol{\pi} \rtimes \rho),
$$ 
where $\boldsymbol{\pi} = \bigoplus_{j=1}^m \pi_j \in S_{nm}$ and the wreath product $\boldsymbol{\pi} \rtimes \rho \in S_n \wr S_m$ is the induced permutation from the block permutation matrix $P_{\boldsymbol{\pi}\rtimes \rho} = \big(\bigoplus_{j=1}^m P_{\pi_j}\big)\cdot (P_\rho \otimes \V I_n)$.}



\begin{figure}
    \centering
    \begin{minipage}[c]{0.2\textwidth} 
        {
        \begin{subfigure}[b]{\textwidth}
            \centering
            \begin{tikzpicture}
            \node[circle, draw] (n2) at (2,1) {2};
            \node[circle, draw] (n1) at (1,2) {1};
            \draw[thick] (n2) -- (n1);  
            \end{tikzpicture}
            \caption{$\mathcal T(12)$}
        \end{subfigure}
        \vspace{0.5cm} 
        \begin{subfigure}[b]{\textwidth}
            \centering
            \begin{tikzpicture}
            \node[circle, draw] (n2) at (2,2) {2};
            \node[circle, draw] (n1) at (1,1) {1};
            \draw[thick] (n2) -- (n1);  
            \end{tikzpicture}
            \caption{$\mathcal T(21)$}
        \end{subfigure}}
    \end{minipage}
    \hfill
    \begin{minipage}[c]{0.3\textwidth} 
        {
        \begin{subfigure}[b]{\textwidth}
            \centering
            \begin{tikzpicture}
            \node[circle, draw] (n2) at (1,1) {2};
            \node[circle, draw] (n1) at (1,3) {1};
            \node[circle, draw] (n3) at (2,2) {3};
            \draw[line width=1mm,blue] (n3) -- (n1);  
            \draw[line width=1mm,blue] (n3) -- (n2);  
            \end{tikzpicture}
            \caption{$\mathcal T(132)$}
        \end{subfigure}}
    \end{minipage}
    \hfill
    \begin{minipage}[c]{0.4\textwidth} 
        {
        \begin{subfigure}[b]{\textwidth}
            \centering
            \begin{tikzpicture}
            \node[circle, draw] (n6) at (4,4) {6};
            \node[circle, draw] (n5) at (3,3) {5};
            \node[circle, draw] (n2) at (3,5) {2};
            \node[circle, draw] (n1) at (2,4) {1};
            \node[circle, draw] (n4) at (3,1) {4};
            \node[circle, draw] (n3) at (2,2) {3};
            \draw[thick] (n2) -- (n1);  
            \draw[thick] (n6) -- (n5);  
            \draw[thick] (n3) -- (n4);  
            \draw[line width=1mm, blue] (n2) -- (n6);  
            \draw[line width=1mm, blue] (n5) -- (n3);  
            \end{tikzpicture}
            \caption{$\mathcal T(216534)$}
        \end{subfigure}}
    \end{minipage}

    \caption{${\mathcal T(132) \, \boxdot\, (\mathcal T(21),\mathcal T(12),\mathcal T(21)) = }\mathcal T(216534)$ formed using $\mathcal T(132)$ with nodes replaced by copies of $\mathcal T(12)$ and $\mathcal T(21)$.}
    \label{fig: wreath BST example}
\end{figure}
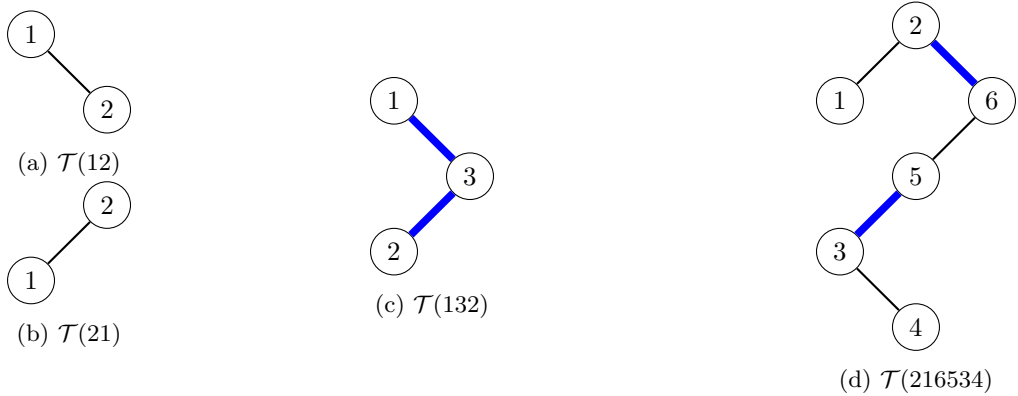

{We study the heights of BSTs generated by permutations sampled uniformly from $S_n \wr S_m$ and $S_m \otimes S_n$. Throughout, we refer to the $S_m$ factor as the \emph{external} permutation model and the $S_n$ factor as the \emph{internal} permutation model. These labels reflect the structure of the associated permutation matrix: the $S_m$ permutation determines the placement of the $n \times n$ blocks within the $nm \times nm$ matrix, while each block is itself an $n \times n$ permutation matrix arising from the $S_n$ component.} 
For example, consider $132 \in S_3$, along with the permutations $21,12,21 \in S_2$, and then $(21 \oplus 12 \oplus 21) \rtimes 132 = (21|34|65) \rtimes 132 = 216534 \in S_2 \wr S_3 \subset S_6$ has corresponding permutation matrix $P_{216534}$ formed by placing copies of $P_{12}$ and $P_{21}$ within $P_{132}$, as seen here:
\begin{align*}
    P_{216534} = \begin{bmatrix}
\mathbf{0}&\mathbf{1}&0&0&0&0\\\mathbf{1}&\mathbf{0}&0&0&0&0\\0&0&0&0&\mathbf{1}&\mathbf{0}\\0&0&0&0&\mathbf{0}&\mathbf{1}\\0&0&\mathbf{0}&\mathbf{1}&0&0\\0&0&\mathbf{1}&\mathbf{0}&0&0
    \end{bmatrix} = \begin{bmatrix}
        P_{21} & \V 0 & \V 0\\\V0 & \V 0& P_{12} \\ \V 0 & P_{21} & \V 0
    \end{bmatrix} = (P_{21} \oplus P_{12} \oplus P_{21})\left(P_{132} \otimes \V I_2\right)
\end{align*}
This block structure materializes then directly within the BST for $\mathcal T(216534)$ as this can be realized as the external BST for $\mathcal T(132)$ being updated by replacing each node with copies of $\mathcal T(12)$ or $\mathcal T(21)$ (with keys shifted according to the corresponding labels from the external BST), as seen in \Cref{fig: wreath BST example}. In particular, we note that while the \textit{nodes} of the external BST are replaced by internal BST copies, the \textit{edges} from the external BST that glue together each of the node BSTs connect along the top-left or top-right edges of the parent external node to the root of the child external node (e.g., in \Cref{fig: wreath BST example}, the external node 3 is the parent for external node 2, so the corresponding internal BST $\mathcal T(65)$ connects along the top-left edge (that ends at 5) to the root of $\mathcal T(34)$; so the external edge connects 5 directly to 3 in the overall block BST). \Cref{fig: Kron BSTs bin} shows a particular instance of a block BST formed using Kronecker products instead, where repeated copies of $\mathcal T(2143)$ are then inserted into the node locations of $\mathcal T(\rho)$ with $\rho \in S_2$ to  form ${\mathcal T(\rho) \boxdot \mathcal T(2143) = }\mathcal T(\rho \otimes 2143)$.

\begin{figure}
    \begin{subfigure}{0.2\textwidth}
        \centering
    \begin{tikzpicture}
    
    \node[circle, draw] (n2) at (2,3) {2};
    \node[circle, draw] (n1) at (1,2) {1};
    \node[circle, draw] (n4) at (3,2) {4};
    \node[circle, draw] (n3) at (2,1) {3};
    
    \draw[thick] (n2) -- (n1);  
    \draw[thick] (n2) -- (n4);  
    \draw[thick] (n4) -- (n3);  
    
    \end{tikzpicture}
    \caption{$\mathcal T(2143)$}
    \end{subfigure}
    \hfill  
    \begin{subfigure}{0.3\textwidth}
        \centering
    \begin{tikzpicture}
    
    \node[circle, draw] (n6) at (4,3) {6};
    \node[circle, draw] (n5) at (3,2) {5};
    \node[circle, draw] (n8) at (5,2) {8};
    \node[circle, draw] (n2) at (2,5) {2};
    \node[circle, draw] (n7) at (4,1) {7};
    \node[circle, draw] (n1) at (1,4) {1};
    \node[circle, draw] (n4) at (3,4) {4};
    \node[circle, draw] (n3) at (2,3) {3};
    
    \draw[thick] (n6) -- (n5);  
    \draw[thick] (n6) -- (n8);  
    \draw[line width=1mm,  blue] (n4) -- (n6);  
    \draw[thick] (n8) -- (n7);  
    \draw[thick] (n2) -- (n1);  
    \draw[thick] (n2) -- (n4);  
    \draw[thick] (n4) -- (n3);  
    
    \end{tikzpicture}
    \caption{$\mathcal T(21436587) = T(12 \otimes 2143)$}
    \end{subfigure}   
    \hfill  
    \begin{subfigure}{0.3\textwidth}
        \centering
    \begin{tikzpicture}
    
    \node[circle, draw] (n6) at (4,5) {6};
    \node[circle, draw] (n5) at (3,4) {5};
    \node[circle, draw] (n8) at (5,4) {8};
    \node[circle, draw] (n2) at (2,3) {2};
    \node[circle, draw] (n7) at (4,3) {7};
    \node[circle, draw] (n1) at (1,2) {1};
    \node[circle, draw] (n4) at (3,2) {4};
    \node[circle, draw] (n3) at (2,1) {3};
    
    \draw[thick] (n6) -- (n5);  
    \draw[thick] (n6) -- (n8);  
    \draw[line width=1mm,  blue] (n5) -- (n2);  
    \draw[thick] (n8) -- (n7);  
    \draw[thick] (n2) -- (n1);  
    \draw[thick] (n2) -- (n4);  
    \draw[thick] (n4) -- (n3);  
    
    \end{tikzpicture}
    \caption{$\mathcal T(65872143) = T(21 \otimes 2143)$}
    \end{subfigure}
    \caption{$\mathcal T(12 \otimes 2143)$ and $\mathcal T(21 \otimes 2143)$ formed from gluing two copies of $\mathcal T(2143)$ together with an extra edge (in \textcolor{blue}{blue}) connecting the top-left edges or top-right paths of each parent BST to the root of the child BST.}
    \label{fig: Kron BSTs bin}
\end{figure}
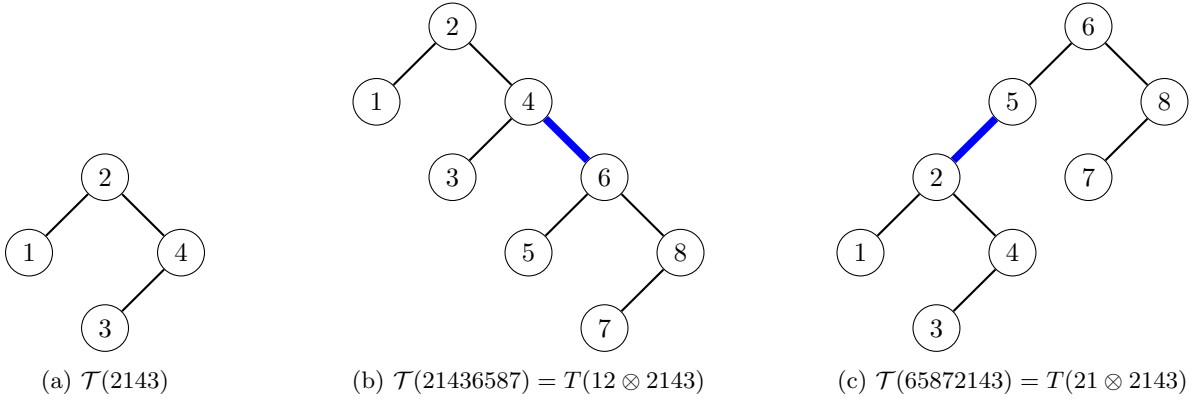

As we are considering the heights of these block BST models, we note these heights do not simply add, due to the dependence on the paths used to reach the leaves across internal and external blocks. So we need to  have a clear understanding of the top-left and top-right traversals taken within each internal BST to get to the leaves of the overall block BST. For example, in \Cref{fig: wreath BST example}, we have a height of 4, that is attained by traversing the 2 external edges along with 2 internal edges. Let $\ell(i), r(i), h(i)$ denote the corresponding top-left edge, top-right edge, and height for the internal BST found in the external node $i$, then we realize $h := h(\mathcal T(216534)) = 4$ as $h = r(1) + 1 + \ell(3) + 1 + h(2)$ (noting further $r(1) = 0$ as the right sub-tree of node 2 is empty for $\mathcal T(21)$). We thus emphasize the more explicit reliance on understanding the behavior of the top-left and top-right edges for the height of the overall block BST, which correspond to the paths from the root to 1 and $n$, respectively. 

We are interested in the asymptotic heights of random block BSTs generated using products of permutations of sizes $n,m$ where $nm$ grows. We first focus on \Cref{thm:1wr}, that studies the heights of $\mathcal T(\pi)$ where $\pi \sim \Unif(S_n \wr S_m)$ or $\pi \sim \Unif(S_m \otimes S_n)$ in the case when $m$ is fixed and $n$ grows unbounded. This is a generalization of \Cref{thm:devroye}, that is thus subsumed by the case $m = 1$.

Note $S_m \otimes S_n \subset S_n \wr S_m$ naturally, as the left hand side comprises the subcases when the internal blocks are identical inside a wreath product structure (and so the semi-direct product action is trivial). Let $\pi \in S_n \wr S_m$. We can further identify the coordinates of $\pi$ by the $m$-node block locations of each set of coordinates, by writing $\pi = (\pi_{\rho(1)} | \pi_{\rho(2)} | \cdots | \pi_{\rho(m)})$, where $\rho \in S_m$ determines the external block structure of $P_\pi$, while $\pi_j$ are each shifted copies of permutations from $S_n$ (i.e., $\pi_j - \min(\pi_j) + 1 \in S_n$) and thus each comprise the internal permutation matrices found within the blocks of $P_\pi$. Hence, $\mathcal T(\pi)$ can be equivalently generated by replacing the nodes of $\mathcal T(\rho)$ with copies of $\mathcal T(\pi_j)$ that are connecting from parent to child blocks from $\mathcal T(\rho)$ along either the top-left or top-right edge of the parent block to the root of the child block (cf. \Cref{fig: wreath BST example,fig: Kron BSTs bin}). In particular, note how the structure of $\mathcal T(\pi_j)$ only cares about the order of the nodes within $\pi_j$ and so is agnostic of shifts of indices; hence $\mathcal T(\pi_j)$ and $\mathcal T(\pi_j - \min(\pi_j) + 1)$ only differ by shifting each label by $-\min(\pi_j) + 1$ (cf. $\mathcal T(2143)$ and $\mathcal T(6587)$ within \Cref{fig: Kron BSTs bin}). We are now prepared to prove \Cref{thm:1wr}.

\begin{proof}[Proof of \Cref{thm:1wr}]

We will only include the proof for the wreath product as the Kronecker product case would follow from the exact same argument adapted to consider identical internal blocks. By the subgroup algorithm (\Cref{thm: subgroup}), independently sampling an element from the outer group $S_n$ and each component of the inner group $S_m$ yields a uniform sample from the full wreath product $S_n \wr S_m$. That is, if $\pi = (\pi_{\rho(1)} | \pi_{\rho(2)} | \cdots | \pi_{\rho(m)}) \sim \Unif(S_n \wr S_m)$ or $\pi \sim \Unif(S_m \otimes S_n)$ with induced block permutations $\rho \in S_m$ and $\pi_j' = \pi_j - \min(\pi_j) + 1 \in S_n$ (i.e., a relabeling of the indices), then $\rho \sim \Unif(S_m)$ is independent of each subblock induced permutation $\pi_j' \sim \Unif(S_n)$.

To compute the height of the block BST formed using $\pi \in S_n \wr S_m$, we need to find the maximal depth of any node within each block BST. Hence, we can first focus on finding the maximal depths within each $\pi_j$, and then finding the maximum among all of these computed maxima. In the case when $m$ is fixed and $n$ grows unbounded (as needed for \Cref{thm:1wr}), this two step process is completely tractable. The first step we will follow shows the depth of the node within a block in the overall tree is determined by the depth of the node inside of the internal permutation plus a shared path from the overall root to the root of the internal node block along the parental heritage paths consisting of top-left or top-right traversals from parent block to subsequent child block root. The next lemma formalizes how a node's depth in the overall block BST decomposes into contributions from internal and external traversals.

\begin{lemma}
    \label{l: 1wr thm step1}
    Let $\pi = (\pi_{\rho(1)} | \cdots | \pi_{\rho(m)}) \in S_n \wr S_m$, where $\rho \in S_m$ and $\pi'_j = \pi_{\rho(j)} - \min(\pi_{\rho(j)}) + 1 \in S_n$ for each $j$. For node $v$ in $\mathcal T(\pi_{\rho(j)}) \subset \mathcal T(\pi)$, let $x = v - \min(\pi_{\rho(j)}) + 1$, and $\V y = (y_1,y_2,\ldots,{\rho(j)})$ the subsequence of $\rho$ that denotes the node sequence in the shortest path in $\mathcal T(\rho)$ connecting the root $\rho(1) = y_1$ to node ${\rho(j)}$. Let $\ell_n^{(j)} = \ell(\pi_j)$ and $r_n^{(j)} = r(\pi_j)$ denote the top-left and top-right edge lengths of $\mathcal T(\pi_j)$, and define
    $$
    G({s,t},u,v) = \left\{ \begin{array}{ll}
    u, & {s > t},\\ v, & {s < t}.
    \end{array}\right.
    $$ Then
    \begin{equation*}
        \depth_{\pi}(v) = \depth_{\pi'_j}(x) +\sum_{k=1}^{\depth_{\rho}(\rho(j))} G(y_k,y_{k+1}, \ell_n^{({y_k})}+1,r_n^{({y_k})}+1).
    \end{equation*}
\end{lemma}
\Cref{l: 1wr thm step1} captures the particular path within $\mathcal T(\pi)$ connecting the root $\pi(1)$ to the node $v$. This path is attained in an explicit form by the path within the corresponding $n$-node BST $\mathcal T(\pi_{\rho(j)})$ along with the path connecting the blocks that travels through either the top-left or top-right edges of the preceding parent $m$-node blocks (determined by the particular path in $\mathcal{T}(\rho)$ connecting the root $\rho(1)$ to the corresponding $m$-node block location $\rho(j)$). Hence, the proof of \Cref{l: 1wr thm step1} is immediate from construction. Note in particular the external $m$-node path is \textit{shared} for all nodes within the same $m$-block. It follows that the maximal depth for any node within the same $m$-block $\rho(j)$ is determined by the height of $\pi_j'$ and the depth of $\rho(j)$:
\begin{corollary}\label{c: step2}
    Let $\pi, \rho,\pi_j'$, $x$, $\V y$ be as in \Cref{l: 1wr thm step1}. Then
    \begin{equation}\label{eq: maxdepth m-node}
        \max_{v \in \mathcal T(\pi_{\rho(j)})} \depth_{\pi}(v) = h(\mathcal T(\pi_j')) + \sum_{k=1}^{\depth_{\rho}(\rho(j))} G(y_k,y_{k+1}, \ell_n^{({y_k})}+1,r_n^{({y_k})}+1).
    \end{equation}
\end{corollary}
Write $X_{n,j,k} := G(y_k,y_{k+1},\ell_n^{(y_k)} + 1, r_n^{(y_k)} + 1)$, where we note additionally $X_{n,j,k} \sim \Upsilon_n$ iid for each $n,j,k$. Conditioning on $\rho$ and using \Cref{c: Stirling limit} along with the fact that $\Upsilon_n \sim \ell_n^{(y_k)} + 1 \sim r_n^{(y_k)}+1$, we have $X_{n,j,k} / \log n$ converges to 1 in probability and $L^p$ for all $p \ge 1$ as $n \to \infty$. Similarly, $h(\mathcal T(\pi_j'))/\log n$ converges to $c^*$ in probability and $L^p$ for all $p \ge 1$, so that the scaled right hand side of \eqref{eq: maxdepth m-node} converges to 
\begin{equation*}
    c^* + \depth_{\rho}(\rho(j)) 
\end{equation*}
in probability (still conditioning on $\rho$) by Slutsky's theorem (\Cref{thm: slutsky}). Now let $\V z \in \mathbb R^m$ be the vector with components $z_j = \max_{v \in \mathcal T(\pi_{\rho(j)})} \depth_\pi(v)$ from which it follows 
\begin{equation*}
    h(\mathcal T(\pi)) = \max_j z_j.
\end{equation*}
In particular, $h(\mathcal T(\pi))$ is a continuous map of $\V z \in \mathbb R^m$. 

If $m$ is fixed, then $h(\mathcal T(\pi_j'))/\log n \to_p c^*$ while $X_{n,j,k}/\log n \to_p 1$ (by \Cref{thm:devroye} and \Cref{c: Stirling limit}). Hence, using the equivalence of limits in probability and in distribution when the limiting term is constant along with the independence of $\pi_j'$ and $\rho$, we see for fixed $j \le m$ and $t$ any non-integer, 
\begin{align*}
    &\lim_{n \to \infty} \P\left(\sum_{k=1}^{\depth_\rho(\rho(j))} X_{n,j,k} \le t \log n \right) \\
    &\hspace{5pc}= \lim_{n \to \infty} \sum_{\ell=1}^{m-1} \P\left({\left.\sum_{k=1}^\ell X_{n,j,k} \le t \log n \, \right|\,}  \depth_\rho(\rho(j)) = \ell  \right) \P(\depth_{\rho}(\rho(j)) = \ell)\\
    &\hspace{5pc}= \sum_{\ell=1}^{m-1}\lim_{n \to \infty}  \P\left(\sum_{k=1}^\ell X_{n,j,k} \le t \log n \right) \P(\depth_{\rho}(\rho(j)) = \ell)\\
    &\hspace{5pc}= \sum_{\ell=1}^{m-1} \P(\ell \le t) \P(\depth_\rho(\rho(j)) = \ell)\\
    &\hspace{5pc}= \sum_{\ell=1}^{m-1} \P(\depth_\rho(\rho(j)) \le t, \depth_\rho(\rho(j)) = \ell)\\
    &\hspace{5pc}= \P(\depth_\rho(\rho(j)) \le t).
\end{align*}
Building on top of these pieces and now also using the continuous mapping theorem (applied to the function $\V z \mapsto \max_j z_j$) and Slutsky's theorem (\Cref{thm: slutsky}), it follows that $h(\mathcal T(\pi))/\log n = \max_j z_j/\log n$ converges in distribution to
\begin{equation*}
    c^* + \max_j \depth_\rho(\rho(j)) = c^* + h(\mathcal T(\rho)).
\end{equation*}
Combining all of these pieces then completes the proof of \Cref{thm:1wr}, noting also that $\log nm = \log n + \log m = \log n \cdot (1 + o(1))$ when $m$ is fixed. 
\end{proof}

A first takeaway from \Cref{thm:1wr} is that  one wreath product or Kronecker product is sufficient to increase the height of a block BST from the uniform case. Thus, \Cref{thm:1wr} is a generalization of \Cref{thm:devroye}, which is now included as a sub-case when $m = 1$ (note $h_1 = 0$). For $m = 2$, we have $h(\mathcal T(\rho)) = 1$ (deterministically), so that $h_{n,m} := h(\mathcal T(\pi))$ limits to $c^* + 1$ when scaled by $\log n$. For $m \ge 3$, $h_m \ge 1$ is random (e.g., $h_3 \sim 1 + \Bern(\frac23)$), so that $c^* + h_m$ is a \textit{random} distributional limit for $h_{n,m}/\log n$ and is strictly larger than $c^*$. Note, though, that the BST height remains proportional to the logarithmic length of the input vector (i.e., $\log nm$).

This presents several natural directions for further study:
\begin{itemize}
    \item What is the asymptotic behavior of $h_{n,m}$ if $m$ grows?
    \item How much can the height of a block BST grow when the input permutation is generated through repeated applications of wreath or Kronecker products? Does this remain $O(\log nm)$?
\end{itemize}
The remainder of this paper, starting in \Cref{sec: butterfly trees}, will focus on addressing the second set of questions. The main set of block BST generating permutations of length $N = 2^n$ we will consider are initialized by $S_2$ and then formed either by using iterated wreath products or Kronecker products. This is the smallest set of (non-trivial) permutations generated using only wreath or Kronecker products.  These align exactly with the nonsimple and simple butterfly permutation groups, $\B_n, \B_{n,s} \subset S_{N}$ (see \Cref{sec: simple,sec: nonsimple}).

For the remainder of the current section, we will partially address the first set of questions. \Cref{thm:1wr} fully answers the questions regarding the asymptotic behavior of $h_{n,m}$ when $m$ is fixed and $n$ grows, so we will next consider the case when $n$ is fixed and $m$ grows. Again, this generalizes \Cref{thm:devroye} in the uniform setting, as this can now be subsumed as the case when $n = 1$. For the general case, we will first establish a lower bound:

\begin{proposition}\label{prop: 1wr fixed n}
    Let $\pi \sim \Unif(S_n \wr S_m)$ or $\pi \sim \Unif(S_m \otimes S_n)$, and $h_{n,m} = h(\mathcal T(\pi))$. Then for any $\varepsilon>0$,
    \begin{equation*}
        \lim_{m \to \infty} \P\left({h_{n,m}} \ge (c^* H_n^{(1)}-\varepsilon)\log m\right) = 1.
    \end{equation*}
\end{proposition}
\begin{proof}
    This is equivalent to finding a specific path in $\pi$ of depth at least $c^* H_n^{(1)} \log m$ asymptotically. Let $\rho \in S_m$ denote the induced block permutation from $\pi$. We will construct such a path by focusing first on an $m$-block of maximal depth in $\rho$, and then choosing a node of maximal depth within this block. Let $\rho(j)$ denote a specific node in $\mathcal T(\rho)$ of maximal depth, i.e., $\depth_\rho(\rho(j)) = h(\mathcal{T}(\rho)) =: h_m$. Let $v$ denote a node of maximal depth in $\mathcal T(\pi_{\rho(j)})$. By \Cref{c: step2} and following the notation from the proof of \Cref{thm:1wr}, it follows that
    $$
    Z_{n,m} := \depth_\pi(v) = h_n + \sum_{k=1}^{h_m} X_{n,j,k}
    $$
    where $X_{n,j,k} \sim \Upsilon_n$ iid for each $k$ when $n,j$ are fixed, and we write $h_n := h(\mathcal T(\pi_{\rho(j)}))$. In particular, we note $h_{n,m} \ge Z_{n,m}$ trivially, while also by the Subgroup algorithm $\rho \sim \Unif(S_m)$ and $\pi_j' \sim \Unif(S_n)$. To establish \Cref{prop: 1wr fixed n}, it now suffices to show $Z_{n,m}/\log m \to_p c^* H_n^{(1)}$ as $m \to \infty$. 
    
    For $n$ fixed, let
    \begin{equation*}
        W_m := \frac1{h_m} \sum_{k = 1}^{h_m} X_{n,j,k}  = \overline {(\Upsilon_n)}_{h_m},
    \end{equation*}
    i.e., a sample mean with $h_m$ iid samples of $\Upsilon_n$ (where we again emphasize that $h_m$ is random). Then 
    $$Z_{n,m} = h_n + h_m W_m.$$

    {If $m \to \infty$, then $h_n/\log m \to_p 0$, 
    $h_m/\log m \to_p c^*$, and we claim $W_m \to_p \E \Upsilon_n = H_n^{(1)}$.
    Using Slutsky's theorem together with the equivalence of convergence in probability and in distribution for constant limits, this yields
    \[
    Z_{n,m}/\log m \to_p c^* H_n^{(1)}
    \]
    for fixed $n$, as desired.}
    

    {To justify $W_m \to_p H_n^{(1)}$, note that $W_m$ is the sample mean of $h_m$ iid copies of $\Upsilon_n$. By the Strong Law of Large Numbers, $\overline{(\Upsilon_n)}_M \to H_n^{(1)}$ almost surely as $M \to \infty$. Now consider any $\omega$ in the almost sure event where $\overline{(\Upsilon_n)}_M(\omega)\to H_n^{(1)}$ as $M \to \infty$. For any $\delta>0$, there exists $M_0(\omega)\geq 1$ such that for any $M\geq M_0(\omega)$, $\big|\overline{(\Upsilon_n)}_M(\omega)- H_n^{(1)}\big|<\delta$. Since $h_m(\omega) \ge \lfloor \log_2 m \rfloor$ and $\lfloor \log_2 m \rfloor \to \infty$, there exists $M_1(\omega)\geq 1$ such that for all $m \ge M_1(\omega)$ we have $h_m(\omega) \ge M_0(\omega)$, and hence $\big|\overline{(\Upsilon_n)}_{h_m(\omega)}(\omega) - H_n^{(1)}\big| < \delta$.  This shows $\overline{(\Upsilon_n)}_{h_m}\to H_n^{(1)}$ almost surely, and therefore $\overline{(\Upsilon_n)}_{h_m}\to_{p} H_n^{(1)}$.   
    }   
%
\end{proof}
We conjecture additionally that $h_{n,m}/\log m \to_p c^* H_n^{(1)}$ as well when $n$ is fixed and $m$ grows. {Establishing a matching upper bound with the same limit would require a more substantial refinement of Devroye’s original argument, which relates heights in BSTs to depths in a fully balanced tree and draws on strong results from extreme-value theory for branching processes.}  Since $H_n^{(1)}/\log n \to 1$ as $n \to \infty$,  this suggests that $h_{n,m}$ scaled by $\log n \log m$ should limit to $c^*$ when both $m,n$ grow:
\begin{conjecture}\label{conj: 1wr}
    If $\pi \sim \Unif(S_n \wr S_m)$ or $\pi \sim \Unif(S_m \otimes S_n)$, then $h_{n,m}/(\log n\log m) \to_p c^*$ as $\min(n,m) \to \infty$.
\end{conjecture}


{At least},  \Cref{thm:1wr} and \Cref{prop: 1wr fixed n} together yield a matching lower bound for \Cref{conj: 1wr}.

\section{Heights of butterfly trees}\label{sec: butterfly trees}

We now address the height of a random BST generated by a permutation of length
$N = 2^n$ formed via iterated wreath and Kronecker products, beginning with $S_2$.
{The resulting trees are called \emph{butterfly trees}, built from the simple
butterfly permutations $\B_{n,s} = S_2^{\otimes n}$ and nonsimple butterfly
permutations $\B_n = S_2^{\wr n}$ (see \Cref{sec: butterfly perms}). Using the
block BST structure of \Cref{sec: block}, they are defined recursively: the base
case $\mathcal{T}_0^{\B}$ is the single-vertex tree, and for each $n \geq 0$,
writing $\mathcal{T}_n^{\B} = \mathcal{T}(\boldsymbol{\pi}_n)$,
$$\mathcal{T}_{n+1}^{\B} = \mathcal{T}(\rho \otimes \boldsymbol{\pi}_n)
= \mathcal{T}(\rho) \, \boxdot \, \mathcal{T}_n^{\B}$$
for $\rho \in S_2$ and $\boldsymbol{\pi}_n \in \B_{n,s}$ in the \emph{simple} case,
and writing $\mathcal{T}_n^{\B} = \mathcal{T}(\boldsymbol{\pi}_n)$,
$\tilde{\mathcal{T}}_n^{\B} = \mathcal{T}(\tilde{\boldsymbol{\pi}}_n)$,
$$\mathcal{T}_{n+1}^{\B} = \mathcal{T}\!\left((\boldsymbol{\pi}_n,
\tilde{\boldsymbol{\pi}}_n) \rtimes \rho\right) = \mathcal{T}(\rho) \, \boxdot \,
\left(\mathcal{T}_n^{\B},\, \tilde{\mathcal{T}}_n^{\B}\right)$$
for $\boldsymbol{\pi}_n, \tilde{\boldsymbol{\pi}}_n \in \B_n$ in the \emph{nonsimple}
case.} As seen in \Cref{thm:1wr}, a \textit{single} wreath or Kronecker product
suffices to increase the asymptotic height from $c^* \approx 4.311$ in the uniform
$S_n$ case to $c^* + h_m$ in the uniform $S_n \wr S_m$ or $S_m \otimes S_n$ case
(scaled by $\log n$, $m$ fixed); we now ask how much further this height grows under
iterated products. Write $h_n^{\B} = h(\mathcal{T}_n^{\B})$ for the height of a
butterfly tree, and similarly $\ell_n^{\B}$ and $r_n^{\B}$ for the lengths of the
top-left and top-right edges of $\mathcal{T}_n^{\B}$, associated to the left-to-right
minima and maxima paths of $\boldsymbol{\pi}_n$.

Notably, this process yields a \textit{unique} BST for each butterfly permutation---a fact easily verified by induction---establishing a bijection between butterfly permutations and butterfly trees. In contrast,  distinct permutations can yield the same BST (e.g., $\mathcal T(213) = \mathcal T(231)$). While group operations on permutations always induce actions on labeled BSTs via $\mathcal T(\boldsymbol{\pi} \cdot \boldsymbol{\pi}') = \mathcal T(\boldsymbol{\pi}) \cdot \mathcal T(\boldsymbol{\pi}')$, the injectivity of $\boldsymbol{\pi} \mapsto \mathcal T(\boldsymbol{\pi})$ in the butterfly setting implies that this action is well defined even at the level of {underlying BST shapes}. That is, butterfly tree shapes themselves inherit the group structure of the permutations via an induced shape-level action.

The nonsimple butterfly permutations—known under various names—form the group $\B_n$, a 2-Sylow subgroup of $S_{2^n}$. As discussed in \Cref{sec: butterfly perms}, their structure and probabilistic properties (see \cite{Abert_Virag_2005, boyd2009fastest, Gamburd_Hoory_Shahshahani_Shalev_Virag_2009}) have been studied extensively in both group-theoretic and random graph contexts. In preceding work \cite{PZ24}, we additionally focused on the questions of the number of cycles and the LIS for butterfly permutations, written respectively as $C(\pi)$ and $\lis(\pi)$ for permutation $\pi$. For simple butterfly permutations, explicit recurrence relations yield precise distributional results:
\begin{theorem}[\cite{PZ24}]\label{thm: old sbt}
    If $\boldsymbol{\pi}_n \sim \Unif(\B_{n,s})$ with $\boldsymbol{\pi}_n = \bigotimes_{j=1}^n \pi_j$ for $\pi_j \in S_2$, then $\frac2N C(\boldsymbol{\pi}_n)-1 = \prod_{j=1}^n \mathds 1(\pi_j = 12) \sim \Bern(\frac1N)$, $\log_2 \lis(\boldsymbol{\pi}_n) = \sum_{j=1}^n \mathds 1(\pi_j = 12) \sim \Binom(n,\frac12)$, and $\lis(\boldsymbol{\pi}_n)\cdot \lds(\boldsymbol{\pi}_n) = N$. Moreover, $\frac2 N C(\boldsymbol{\pi}_n) \to_p 1$ and $(\log_2 \lis(\boldsymbol{\pi}_n) - n/2)/\sqrt{n/4} \to_d N(0,1)$ as $n \to \infty$.
\end{theorem}
For nonsimple butterfly permutations, similar but more intricate recursions arise. The LIS obeys a branching-type recurrence, while the cycle count admits a full (non-asymptotic and asymptotic) distributional description:
\begin{theorem}[\cite{PZ24}]\label{thm: old bsbt}
    If $\boldsymbol{\pi}_n \sim \Unif(\B_n)$, then $\lis(\boldsymbol{\pi}_n) \sim X_n$ follows a distributional recurrence defined by $X_0=1$ and $X_{n+1} \sim (X_n + X_n')\eta_n + \max(X_n,X_n')(1 - \eta_n)$ for iid $X_n \sim X_n'$ independent of $\eta_n \sim \Bern(\frac12)$ for $n \ge 0$. Additionally, there exist constants $\alpha = \log_2(3/2) \approx 0.58496$ and $\beta \approx 0.83255$ such that $N^\alpha \le \E \lis(\boldsymbol{\pi}_n) \le N^\beta$. Furthermore, $C(\boldsymbol{\pi}_n) \sim Y_n$ follows a distributional recurrence defined by $Y_0 = 1$ and $Y_{n+1} \sim Y_n + \eta_n Y_n'$ for iid $Y_n \sim Y_n'$ independent of $\eta_n \sim \Bern(\frac12)$ for $n\geq 0$. Lastly, $C(\boldsymbol{\pi}_n)/N^\alpha \to_d W$, where $W$ takes values in $[0,\infty)$ and is uniquely determined by its positive integer moments $m_k = \E W^k$: these moments satisfy $m_1 = 1$ and
    $$m_{k} = \frac{\lambda - 1}{\lambda^k - 1} \sum_{j=1}^{k-1} \binom{k}j m_j m_{k-j}$$
    for $k\geq 2$, where $\lambda = 2^\alpha = 3/2$.
\end{theorem}
\begin{remark}
As shown in \Cref{thm: old bsbt}, both the LIS and number of cycles for permutations generated via iterated wreath products exhibit significantly larger average values than under uniform permutations. Specifically, the average LIS and number of cycles increase from orders $N^{1/2}$ and $\log N$, respectively, to at least $N^\alpha \approx N^{0.58496}$. For the number of cycles, this shift from logarithmic to polynomial growth precisely mirrors the corresponding increase in height for BSTs built from iterated wreath products versus those built from uniform permutations (see \Cref{thm: nsbt}).
\end{remark}


{In this work, we adapt and extend} the analytical techniques from \cite{PZ24} to study the height \( h_n^{\B} \) of butterfly trees. We further establish connections between the {top-}left and {top-}right edge lengths (\(\ell_n^{\B}, r_n^{\B} \)) and the classical permutation statistics \( C(\boldsymbol\pi) \) and \( \lis(\boldsymbol\pi) \) (the main focus of \cite{PZ24}), which form the core of our approach to deriving distributional properties of butterfly tree heights. By further leveraging the recursive structure of butterfly trees, we derive exact distributional results for the height in the simple case and establish strong asymptotic bounds in the nonsimple case.

{
Of particular interest is the contrast with the uniform case (\Cref{sec: uniform}), where the expected BST height grows logarithmically. As shown in \Cref{thm:1wr}, this logarithmic behavior persists even after applying a single wreath or Kronecker product. However, for butterfly permutations formed by iterated products, the expected BST height grows polynomially, with \(\E h_n^{\B} = \Omega(N^\alpha)\) (see \Cref{thm: sbt,thm: nsbt}). Notably, this polynomial growth exceeds also the \(\Theta(\sqrt{n})\) average height scaling of random rooted planar trees.
}

\subsection{Simple butterfly trees: \texorpdfstring{\boldmath$\mathcal{T}^{\B}_{n,s}$}{TnB}}\label{sec: simple}

We will first consider simple butterfly trees, $\mathcal T^{\B}_{n,s}$, generated using uniform simple butterfly permutations $\boldsymbol{\pi}_n = \bigotimes_{j=1}^n \pi_j \sim \Unif(\B_{n,s})$. \Cref{fig: Kron BSTs bin} shows two instances of simple butterfly trees, since we further note $2143 = 12 \otimes 21$. Hence, $21436587 = 12 \otimes 12 \otimes 21$ and $65872143 = 21 \otimes 12 \otimes 21$ are each simple butterfly permutations from $\B_{3,s}$. When building up $\B_{n,s}$ using Kronecker products of either 12 or 21, then we can write $\boldsymbol{\pi}_{n + 1} = \pi_{n+1} \otimes \boldsymbol{\pi}_n \in \B_{n+1,s}$ recursively as a Kronecker product of $\pi_{n+1} \in S_2 = \B_{1,s}$ and $\boldsymbol{\pi}_n \in \B_{n,s}$ (cf. \Cref{fig: Kron BSTs bin}).

Similar to the LIS question for simple butterfly permutations, \Cref{thm: sbt} shows we again are able to attain exact nonasymptotic and asymptotic distributional descriptions for the heights of simple butterfly trees, $h_n^{\B}$. The proof of \Cref{thm: sbt}  follows from a series of recursive relationships relating $\boldsymbol{\pi}_{n+1} = \pi_{n+1} \otimes \boldsymbol{\pi}_n$ directly to properties of $\pi_{n+1}$ and $\boldsymbol{\pi}_n$, which we will establish through a series of lemmas. 

We begin by highlighting a structural property of simple butterfly trees: their height decomposes exactly into the sum of the lengths of their top-left and top-right edge paths. This clean decomposition reflects a recursive geometry that will later connect to the LIS and LDS (longest decreasing subsequence) of $\boldsymbol{\pi}_n$. Namely, we have:

\begin{lemma}\label{l:2}
    Let $\boldsymbol{\pi}_n \in \B_{n,s}$. Then $h(\mathcal T(\boldsymbol{\pi}_n)) = \ell(\mathcal T(\boldsymbol{\pi}_n))+r(\mathcal T(\boldsymbol{\pi}_n))$.
\end{lemma}

\begin{proof}
    For the sake of compactness, write $h_n^{\B} = h(\mathcal T(\boldsymbol{\pi}_n))$, $\ell_n^{\B} = \ell(\mathcal T(\boldsymbol{\pi}_n))$, and $r_n^{\B} = r(\mathcal T(\boldsymbol{\pi}_n))$. We will establish this by induction on $n$ for $\boldsymbol{\pi}_n \in \B_{n,s}$. The $n=1$ case is immediate, since $r_1^{\B} = \mathds 1(\pi_1 = 12)$ and $\ell_1^{\B} = \mathds 1(\pi_1 = 21) = 1 - r_1^{\B}$ as each can be either $0$ or $1$ depending on the value $\pi_1 \in S_2$ (see \Cref{fig: wreath BST example}(a-b) for instances of $\mathcal T(12)$ and $\mathcal T(21)$). Hence, noting the height is deterministically equal to 1 in either case, we {thus} have $h_1^{\B} =1= \ell_1^{\B} + r_1^{\B}$. 
    
    Now assume the result holds for $\B_{m,s}$ for $m \le n$, and let $\boldsymbol{\pi}_{n+1} \in \B_{n+1,s}$. Let $\pi_{n+1} \in S_2$ and $\boldsymbol{\pi}_n \in \B_{n,s}$ be such that $\boldsymbol{\pi}_{n+1} = \pi_{n+1} \otimes \boldsymbol{\pi}_n$. First, consider the case when $\pi_{n+1} = 12$, then $\mathcal T(\boldsymbol{\pi}_{n+1})$ is constructed by gluing together two copies of $\mathcal T(\boldsymbol{\pi}_n)$ connecting from the bottom of the top-right edge of the {parent} copy of $\mathcal T(\boldsymbol{\pi}_n)$ to the root of the right-child copy of $\mathcal T(\boldsymbol{\pi}_n)$ (i.e., the sub-tree $\mathcal T(\boldsymbol{\pi}_n + N)$). Hence, the top-right edge of $\mathcal T(\boldsymbol{\pi}_{n+1})$ is comprised of the top-right edges of the two copies of $\mathcal T(\boldsymbol{\pi}_n)$ along with the extra edge gluing together the two sub-trees (i.e., $r_{n+1}^{\B}= 2 \cdot r_n^{\B} + 1$), while the top-left edge is comprised of only the top-left edge of the parent copy of $\mathcal T(\boldsymbol{\pi}_n)$ (i.e., $\ell_{n+1}^{\B} = \ell_n^{\B}$). In particular, we note the height of $\mathcal T(\boldsymbol{\pi}_{n+1})$ is necessarily determined by the right-sub-tree of the root, and is thus determined by the path traveling the top-right edge of the parent copy, passing over the gluing edge, followed by a maximal path in the lower child (i.e., $h_{n+1}^{\B} = r_n^{\B} + 1 + h_n^{\B}$). {Hence, since $r_{n+1}^{\B} = 2  \cdot r_n^{\B} + 1$ and $\ell_{n+1}^{\B} = \ell_n^{\B}$, we have by the inductive hypothesis
    \[
    h_{n+1}^{\B} = r_n^{\B} + 1 + h_n^{\B} = r_n^{\B} + 1 + (r_n^{\B} + \ell_n^{\B}) = (2 \cdot r_n^{\B} + 1) + \ell_n^{\B} = r_{n+1}^{\B} + \ell_{n+1}^{\B}.
    \]
    The result for $\pi_{n+1} = 21$ follows analogously.}  
\end{proof}

{
\begin{remark}    
    In the proof of \Cref{l:2}, we have
    \begin{align*}
        \ell_{n+1}^{\B} &= \left\{\begin{array}{ll}
        \ell_{n}^{\B}, & \pi_{n+1} = 12\\
        2 \cdot \ell_{n}^{\B} + 1, \hspace{1pc}& \pi_{n+1} = 21
        \end{array}\right.\\
        r_{n+1}^{\B} &= \left\{\begin{array}{ll}
        2 \cdot  r_{n}^{\B} + 1,\hspace{1pc} & \pi_{n+1} = 12\\
        r_{n}^{\B}, & \pi_{n+1} = 21\\
        \end{array}\right.\\
        h_{n+1}^{\B} &= \left\{\begin{array}{ll}
        r_{n}^{\B} + 1 + h_{n}^{\B}, & \pi_{n+1} = 12\\
        \ell_{n}^{\B} + 1 + h_{n}^{\B}, & \pi_{n+1} = 21\\
        \end{array}\right.
    \end{align*}
    This can be condensed into defining the recursion on the vector $(h_n^{\B},\ell_n^{\B},r_n^{\B})$ as
\begin{align}
(h_{n+1}^{\B},\ell_{n+1}^{\B},r_{n+1}^{\B}) 
    &= (h_{n}^{\B},\ell_{n}^{\B},r_{n}^{\B}) + \left\{
    \begin{array}{ll}
        (r_n^{\B}+1)(1,0,1), & \pi_{n+1} = 12\\ \vspace{-.5pc}\\
        (\ell_n^{\B}+1)(1,1,0), & \pi_{n+1} = 21
    \end{array}
    \right.\label{eq: Hn s rec}
\end{align}
So the preceding proof can equivalently be summarized by observing inductively that $(h_{n}^{\B},\ell_{n}^{\B},r_{n}^{\B})$ lies in $ \operatorname{Span}\{(1,0,1),(1,1,0)\}$; since $(x,y,z) \in \{(1,0,1),(1,1,0)\}$ satisfies $x = y+z$, it follows that $h_n^{\B} = r_n^{\B} + \ell_n^{\B}$.
\end{remark}
}

\begin{figure}[t]
    \centering
    \includegraphics[width=0.325\linewidth]{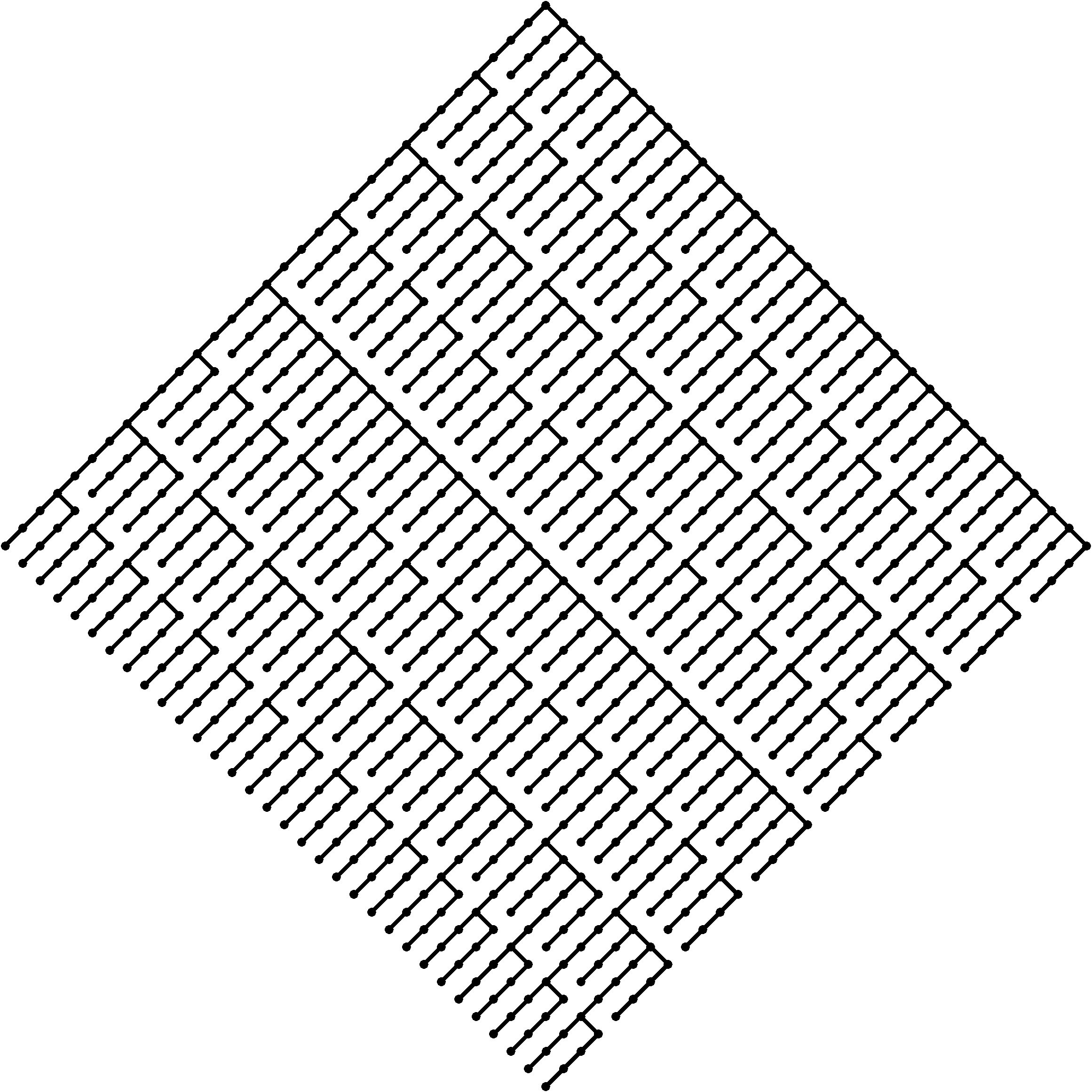}
    \includegraphics[width=0.325\linewidth]{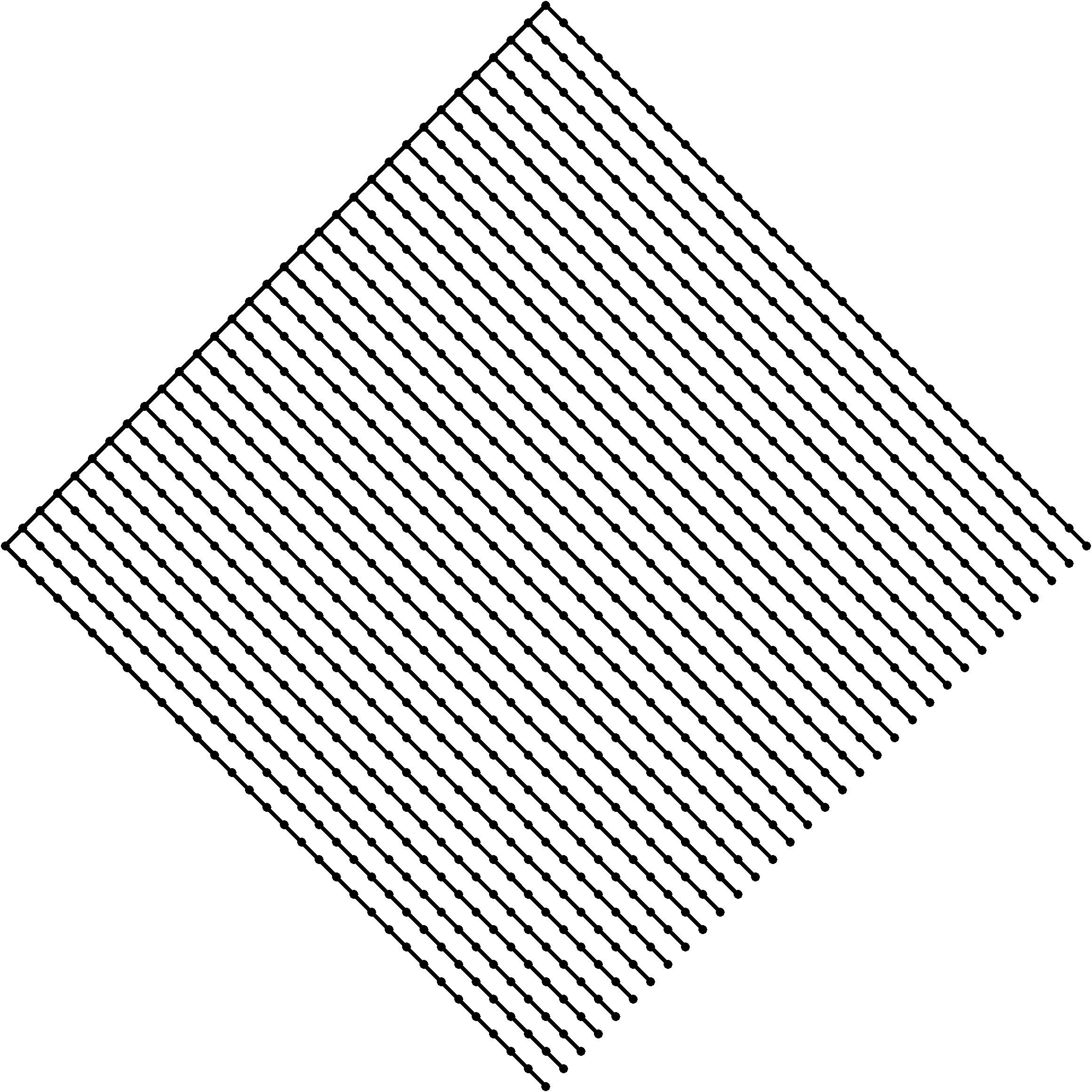}
    \includegraphics[width=0.325\linewidth]{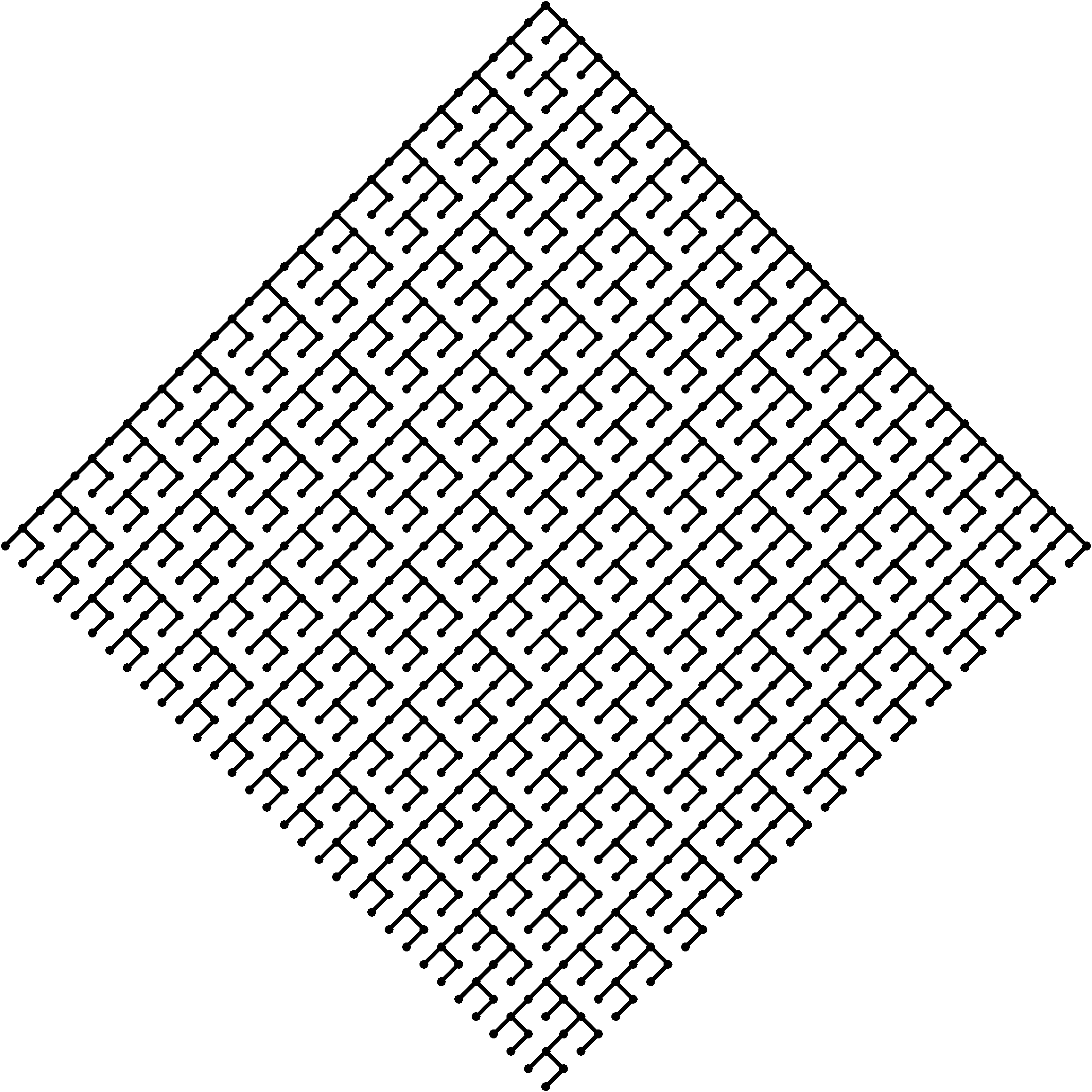}
    \caption{Simple butterfly trees of minimal height with $N = 2^{10} = 1{,}024$ nodes}
    \label{fig:minimal height BSTs}
\end{figure}

\begin{remark}\label{rmk: rectangular}
    Note that \Cref{l:2} admits a more direct intuitive {geometric} interpretation when analyzing the graph associated with the butterfly tree: each simple butterfly tree necessarily fills a rectangular grid of \( N = 2^n \) nodes, with the node of maximal depth located precisely at the bottom corner of this grid. These structures are illustrated in \Cref{fig: Kron BSTs bin,fig:minimal height BSTs}.
\end{remark}

\begin{remark}
    \Cref{l:2} is in stark contrast to the uniform permutation case $\pi \sim \Unif(S_n)$, where each height and left-to-right maxima/minima scale as $\Theta(\log n)$ with differently limiting coefficients, so that $h_n - (\ell_n + r_n) \gtrsim (c^*-2) \log n + o(\log n) = \omega(1)$ typically.
\end{remark}

 
{
By \Cref{thm: old sbt}, we have 
\[
N = \lis(\boldsymbol{\pi}_n) \cdot \lds(\boldsymbol{\pi}_n),
\]
with 
\[
\log_2 \lis(\boldsymbol{\pi}_n) = \sum_{j=1}^n \mathds1(\pi_j = 12) \sim \operatorname{Binom}(n,1/2).
\] 
Consequently, $\log_2 \lds(\boldsymbol{\pi}_n) = n - \log_2 \lis(\boldsymbol{\pi}_n)$. This can also be shown directly (as seen in \cite{PZ24}) by viewing the LIS as a down-right path of 1's in the permutation matrix and the LDS as an up-right path. Inductively, in the iterated Kronecker product, each step either preserves or doubles the LIS and LDS. The same recursion applies to \(r(\mathcal T(\boldsymbol{\pi}_n))+1\) and \(\ell(\mathcal T(\boldsymbol{\pi}_n))+1\), as confirmed by the above inductive construction. In particular, the left-to-right maxima and minima coincide with the LIS and LDS, respectively, for a simple butterfly permutation. This is summarized in:
}


\begin{lemma}\label{l: simple lis r}
    Let $\boldsymbol{\pi}_n \in \B_{n,s}$. Then $\lis(\boldsymbol{\pi}_n) = r(\mathcal T(\boldsymbol{\pi}_n))+1$ and $\lds(\boldsymbol{\pi}_n) = \ell(\mathcal T(\boldsymbol{\pi}_n))+1$.
\end{lemma}

Hence, together this yields the first statement from \Cref{thm: sbt}:
\begin{equation}\label{eq: h sb form}
    h_n^{\B} \sim 2^{X_n} + 2^{n-X_n} - 2, \qquad X_n \sim \operatorname{Binom}(n,1/2){.}
\end{equation}
Since $\E 2^{X_n} = 2^{-n}\sum_{k=0}^n 2^k\binom{n}k = (3/2)^n$ and $X_n \sim n-X_n$, it follows immediately 
\begin{equation*}
    \E h_n^{\B} = 2\left(\frac32\right)^n - 2 = 2 \cdot N^\alpha \cdot (1 + o(1)), \qquad \alpha = \log_2(3/2) \approx 0.58496.
\end{equation*}

From \eqref{eq: h sb form}, $h_n^{\B}$ has support among $1 + \lfloor n/2\rfloor$ values, with minimum $2^{\lfloor n/2\rfloor} + 2^{\lceil n/2\rceil}-2=\Theta(\sqrt N)$ and maximum $N -1$. \Cref{t: s height cnts} shows the total counts of the $N = 2^{10} = 1{,}024$ simple butterfly trees with $N$ nodes that match each possible height. \Cref{fig:minimal height BSTs} shows three different simple butterfly trees with \( N = 2^{10} = 1{,}024 \) nodes and minimal height. This minimal height is necessarily achieved when \( \boldsymbol{\pi} \in \B_{10,s} \) is formed as a Kronecker product of an equal number of copies of the permutations \(12\) and \(21\). The middle butterfly tree in \Cref{fig:minimal height BSTs} is generated using the permutation \( 21 \otimes \cdots \otimes 21 \otimes 12 \otimes \cdots \otimes 12 \).
\begin{table}[h!]
\centering
\begin{tabular}{c|c|c}
$k$ & $f(k) = 2^k + 2^{10-k} - 2$ & $\#\{\boldsymbol{\pi} \in \B_{10,s}: h(\mathcal T(\boldsymbol \pi)) = f(k)\}$ \\
\hline
0 & 1023 & 2 \\
1 & 512  & 20 \\
2 & 258  & 90 \\
3 & 134  & 240 \\
4 & 78   & 420 \\
5 & 62   & 252 \\
\end{tabular}
\caption{Height counts associated to the $N=2^{10} = 1{,}024$ simple butterfly permutations in $\B_{10,s}$}
\label{t: s height cnts}
\end{table}

\begin{remark}
    Recall in the uniform permutation case $\pi_n \sim \Unif(S_n)$,  we have  $h(\mathcal T(\pi_n))/\log n \to_p c^* \approx 4.311$ (the unique solution to $x \log(2e/x) = 1$ for $x \ge 2$) while $\operatorname{LIS}(\pi_n)/\sqrt n \to_p 2$. In particular, $h(\mathcal T(\pi_n))$ is typically much smaller than $\operatorname{LIS}(\pi_n)$. For the simple binary butterfly permutation case, this is a different story, since now $\E h(\mathcal T(\boldsymbol{\pi}_n)) = \Theta(\E \lis(\boldsymbol{\pi}_n))$ by \Cref{l: simple lis r} and \eqref{eq: h sb form}.
\end{remark}

\begin{remark}\label{rmk: s rectangle}
    We note the distribution given in \eqref{eq: h sb form} matches exactly the degree distribution of a uniformly sampled vertex from the comparability graph of the Boolean lattice with $n$-element ground set. \Cref{fig:comparability graph} displays the  adjacency matrix for this graph. Remarkably, its sparsity pattern coincides with that of $L+U$, where $PB = LU$ is {the} GEPP factorization of a simple butterfly matrix $B$---up to the diagonal entries, which are zero in the adjacency matrix and nonzero in $L+U$. (Interestingly, the visualization appears to manifest several butterfly-like forms.) In this factorization, the permutation matrix factor $P$ corresponds to a simple butterfly permutation. Future work can explore additional connections between random graphs and GEPP.
\end{remark}

\begin{figure}[t]
    \centering
    \includegraphics[width=0.5\linewidth]{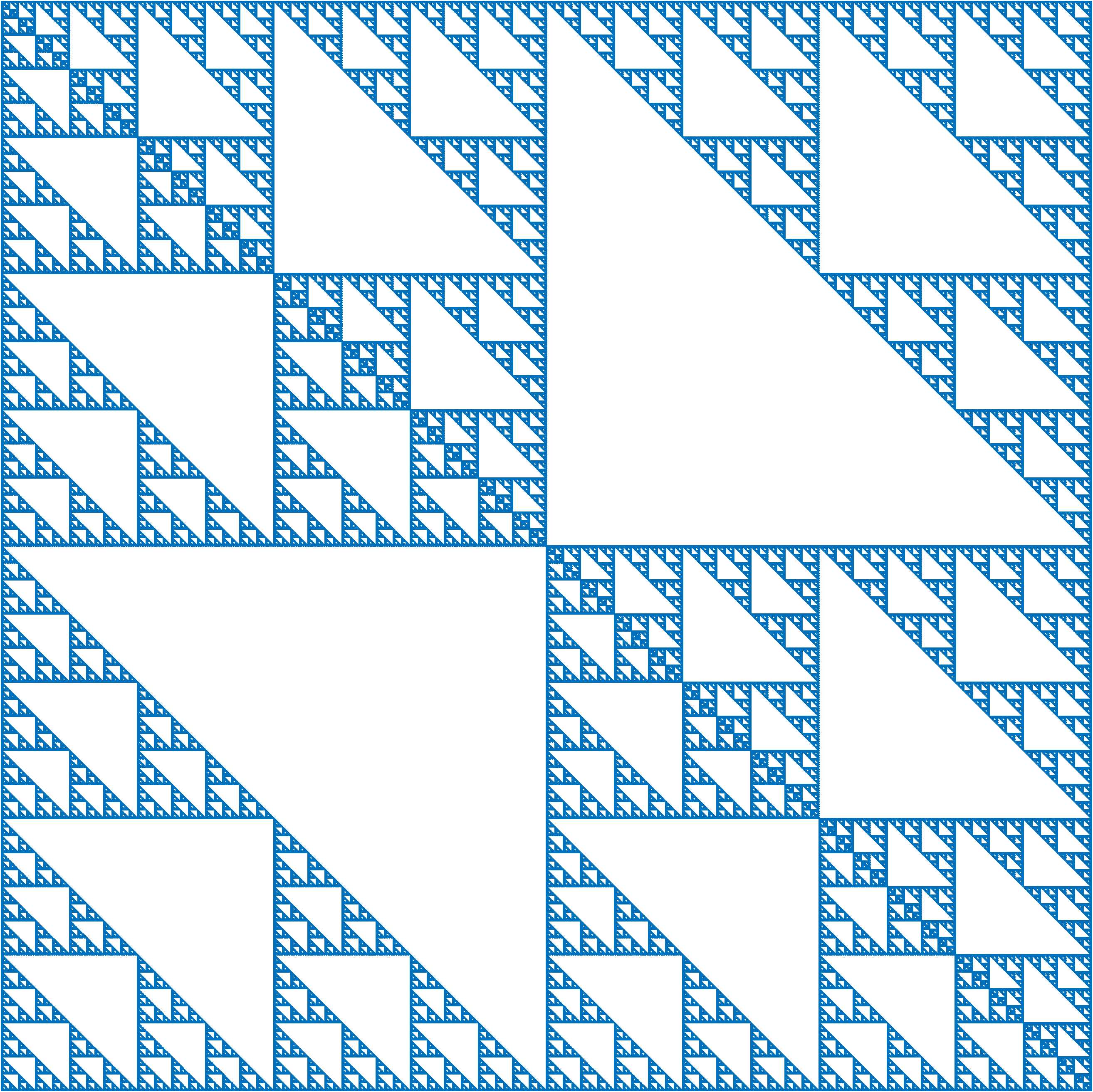}
    \caption{Sparsity pattern map for the adjacency matrix for the comparability graph of the Boolean lattice of size $2^{12}$}
    \label{fig:comparability graph}
\end{figure}

The remaining part of \Cref{thm: sbt} requires establishing a distributional limit for $(\log_2 h(\mathcal T(\pi_n)) - n/2)/(\sqrt n /2)$ to $|Z|$ for $Z \sim N(0,1)$ as $n \to \infty$ for uniform simple butterfly permutations $\pi_n$. This follows from standard results from probability theory: 



Let  $Z_n = 2(X_n - n/2)/\sqrt n$ with $Z_n \to_d Z\sim N(0,1)$ by the Central Limit Theorem (CLT), with $h_n^{\B},X_n$ defined as in \eqref{eq: h sb form}. Let $R_n = h_n^{\B} + 2$. Then $$W_n := 2^{-n/2} R_n = 2^{X_n - n/2} + 2^{n/2 - X_n}   = 2^{\sqrt n |Z_n|/2}(1 + 2^{-\sqrt n |Z_n|}),$$ and now let $$Q_n = W_n^{2/\sqrt n} := A_n B_n^{2/\sqrt n}$$ where $A_n = 2^{|Z_n|}$ and $B_n = 1 + 2^{-\sqrt n |Z_n|}$. Note first $B_n \to_p 1$. This is 
{straightforward}, 
but to include the details: Fix $\epsilon > 0$, and note $$\P\left(|B_n - 1| = 2^{-\sqrt n|Z_n|} > \epsilon\right) = \P\left(|Z_n| < \frac{\log_2 \frac1\epsilon}{\sqrt n}\right){.}$$
If $\epsilon \ge 1$, then this evaluates directly to 0 since $|Z_n| \ge 0$ while $\log_2 \frac1\epsilon \le 0$, so we only need to consider $\epsilon \in (0,1)$. Fix $\delta > 0$ and let $N_\delta$ be such that $\frac{\log_2 \frac1\epsilon}{ \sqrt n} < \delta$ if $n \ge N_\delta$. Then we have
$$
\P(|B_n - 1| > \epsilon) \le \P(|Z_n| < \delta)
$$
for $n \ge N_\delta$, so that
$$
\limsup_{n \to \infty} \P(|B_n - 1| > \epsilon) \le \P(|Z| < \delta) = 2 \Phi(\delta) - 1 = o(\delta)
$$
by the CLT, with $\Phi$ denoting the cumulative distribution function for $Z$. It follows that $B_n \to_p 1$ by taking $\delta \to 0$. In particular, we have $B_n^{2/\sqrt n} \to_p 1$. 
Since $A_n \to_d 2^{|Z|}$ (by the continuous mapping theorem), {by Slutsky's theorem it follows} that $$Q_n = A_n B_n^{2/\sqrt n} \to_d 2^{|Z|},$$ where we recall $$Q_n = W_n^{2/\sqrt n} = (2^{-n/2} R_n)^{2/\sqrt n} = [2^{-n/2} (h_n^{\B}+2)]^{2/\sqrt n}.$$
Hence, by the continuous mapping theorem, we have $$\log_2 Q_n = 2(\log_2 (h_n^{\B} + 2) - n/2)/\sqrt n \to_d |Z|.$$
Replacing $h_n^{\B} + 2$ directly with just $h_n^{\B}$ above {preserves the} distributional limit (via Slutsky's theorem), i.e., 
$$
2(\log_2 h_n^{\B} - n/2)/\sqrt n \to_d |Z|.
$$
(This mirrors the same normalization that determined $Z_n$ from $X_n$ for the CLT application above.) 
Combining all of these results now completes the proof of \Cref{thm: sbt}.

\begin{remark}
    A direct consequence of \Cref{thm: sbt} along with {\Cref{rmk: rectangular}} is that we now also have the full joint distribution for the widths of simple butterfly trees at depth $k$, $$W_k(\boldsymbol \pi_n) = \#\{v \in \mathcal T(\boldsymbol \pi_n): \depth_{\boldsymbol \pi_n}(v) = k\}$$ for $\boldsymbol{\pi}_n \sim \Unif(\B_{n,s})$ and hence also on the maximal width
    \begin{equation}\nonumber
        W(\boldsymbol \pi_n) = \max_{k} W_k(\boldsymbol{\pi}_n).
    \end{equation}
    In particular, 
    \begin{align*}\nonumber
        W(\boldsymbol{\pi}_n) &= W_{\min(\lis(\boldsymbol \pi_n),\lds(\boldsymbol \pi_n))-1}(\boldsymbol \pi_n) = \min(\lis(\boldsymbol \pi_n),\lds(\boldsymbol \pi_n)) \sim 2^{\min(X_n,n-X_n)}
    \end{align*}
    for $X_n \sim \operatorname{Binom}(n,1/2)$, using also \Cref{l: simple lis r} and \Cref{thm: old sbt}. Since $\lis(\boldsymbol{\pi}_n) \cdot \lds(\boldsymbol{\pi}_n) = N$ (from \Cref{thm: old sbt}), the length of the longest monotone subsequence for a simple butterfly permutation, $$M(\boldsymbol{\pi}_n) = \max(\lis(\boldsymbol{\pi}_n),\lds(\boldsymbol{\pi}_n))$$ satisfies
    $$
    W(\boldsymbol{\pi}_n) \cdot M(\boldsymbol{\pi}_n) = N.
    $$
\end{remark}

\subsection{Nonsimple butterfly trees: \texorpdfstring{\boldmath$\mathcal{T}^{\B}_n$}{TnB}}\label{sec: nonsimple}

Now we turn to nonsimple butterfly trees, $\mathcal T_n^{\B}$, generated by taking $\boldsymbol \pi \sim \Unif(\B_n)$. Unlike the simple case, where permutations are built via Kronecker products such as $12 \otimes \pi_1 = \pi_1 \oplus \pi_1$ and $21 \otimes \pi_1 = \pi_1 \ominus \pi_1$, the nonsimple setting involves wreath product constructions. These allow distinct subblocks—e.g., $\pi_1 \oplus \pi_2$ and $\pi_1 \ominus \pi_2$ with $\pi_1 \ne \pi_2$. This added structural flexibility {complicates the analysis of} the height $h_n^{\B} = h(\mathcal T(\boldsymbol \pi))$. As a result, a full distributional characterization (cf.\ \Cref{thm: sbt}) is {no longer fully} tractable. Instead, we focus on bounding the expected height and show that $$cN^{\alpha}\cdot (1+o(1))\leq {\E}[h_n^B]\leq dN^{\beta}\cdot (1+o(1))$$ for $\beta \approx 0.913189$, as detailed in \Cref{thm: nsbt}. {We emphasize that the significance of the upper bound in \Cref{thm: nsbt} lies in the sublinear growth, rather than in the precise value of the exponent $\beta$.}

\begin{figure}
    \begin{subfigure}{0.4\textwidth}
        \centering
    \begin{tikzpicture}
    
    \node[draw] (n2) at (1,2) {$\mathcal T(\pi_1)$};
    \node[draw] (n1) at (2,1) {$\mathcal T(\pi_2)$};
    
    \draw[line width=1mm,  blue] (n2) -- (n1);  
    
    \end{tikzpicture}
    \caption{$\mathcal T(\pi_1 \oplus \pi_2)$}
    \end{subfigure}
    \hfill
    \begin{subfigure}{0.4\textwidth}
        \centering
    \begin{tikzpicture}
    
    \node[draw] (n2) at (1,1) {$\mathcal T(\pi_2)$};
    \node[draw] (n1) at (2,2) {$\mathcal T(\pi_1)$};
    
    \draw[line width=1mm,  blue] (n2) -- (n1);  
    
    \end{tikzpicture}
    \caption{$\mathcal T(\pi_1 \ominus \pi_2)$}
    \end{subfigure}

    \vspace{1mm}

    \centering
    \begin{subfigure}{0.8\textwidth}
        \centering
    \begin{tikzpicture}
    
    \node[circle, draw] (n6) at (4,5) {6};
    \node[circle, draw] (n5) at (3,4) {5};
    \node[circle, draw] (n8) at (5,4) {8};
    \node[circle, draw] (n2) at (2,3) {2};
    \node[circle, draw] (n7) at (4,3) {7};
    \node[circle, draw] (n1) at (1,2) {1};
    \node[circle, draw] (n4) at (3,2) {4};
    \node[circle, draw] (n3) at (2,1) {3};

    \node[circle, draw] (n9) at (5,6) {9};
    \node[circle, draw] (n10) at (6,5) {10};
    \node[circle, draw] (n11) at (7,4) {11};
    \node[circle, draw] (n12) at (8,3) {12};
    \node[circle, draw] (n13) at (9,2) {13};
    \node[circle, draw] (n14) at (10,1) {14};
    \node[circle, draw] (n15) at (11,0) {15};
    \node[circle, draw] (n16) at (12,-1) {16};
    
    \draw[thick] (n6) -- (n5);  
    \draw[thick] (n6) -- (n8);  
    \draw[thick] (n5) -- (n2);  
    \draw[thick] (n8) -- (n7);  
    \draw[thick] (n2) -- (n1);  
    \draw[thick] (n2) -- (n4);  
    \draw[thick] (n4) -- (n3);  

    \draw[thick] (n9) -- (n10);
    \draw[thick] (n10) -- (n11);
    \draw[thick] (n11) -- (n12);
    \draw[thick] (n12) -- (n13);
    \draw[thick] (n13) -- (n14);
    \draw[thick] (n14) -- (n15);
    \draw[thick] (n15) -- (n16);

    \draw[line width=1mm,  blue] (n6) -- (n9);  
    \end{tikzpicture}
    \caption{$T(12345678 \ominus 65872143)$}
    \end{subfigure}
    \hfill
    \caption{$\mathcal T(\boldsymbol\pi)$ for $\boldsymbol\pi = (\rho \otimes \operatorname{id})(\pi_1 \oplus \pi_2)$ for $\rho = 12$ (in (a)) and $\rho = 21$ (in (b)), while (c) shows a particular instance of a nonsimple butterfly tree of the form $\mathcal T(\pi_1 \ominus \pi_2)$ where $\pi_1 \ne \pi_2$}
    \label{fig: ns bin tree}
\end{figure}
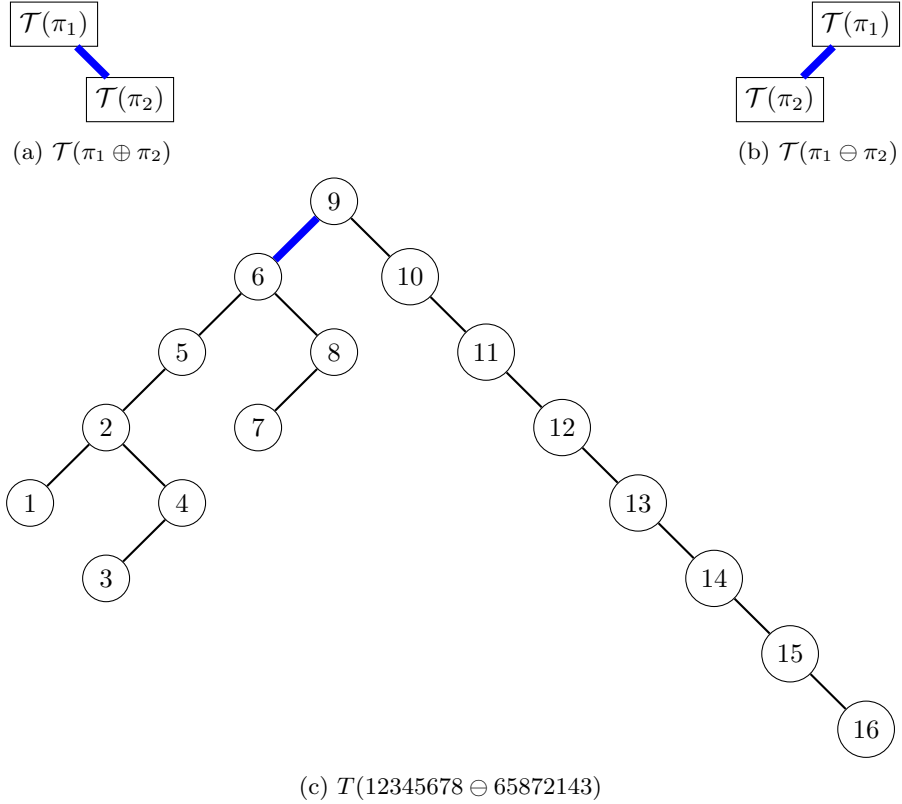
The BST $\mathcal T(\boldsymbol \pi)$ can be recursively constructed by replacing each node in an external BST (viz., $\mathcal T(12)$ or $\mathcal T(21)$) with a subtree generated from a butterfly permutation (with shifted keys). \Cref{fig: ns bin tree} illustrates this recursive construction. The top row shows $\mathcal T(\pi_1 \oplus \pi_2)$ and $\mathcal T(\pi_1 \ominus \pi_2)$, corresponding to the two choices of $\rho \in S_2$. The bottom figure gives a specific example of $\mathcal T(\pi_1 \ominus \pi_2)$ where $\pi_1 \ne \pi_2$. To analyze the height, we consider $\boldsymbol \pi = (\rho \otimes \operatorname{id})(\pi_1 \oplus \pi_2)$ for $\pi_1, \pi_2 \in \B_n$ and $\rho \in S_2$, and define 
$$H_{n+1} = h(\mathcal T(\boldsymbol \pi)), \quad H_n = h(\mathcal T(\pi_1)), \quad H_n' = h(\mathcal T(\pi_2)),$$ with $L_n, L_n', L_{n+1}$ and $R_n, R_n', R_{n+1}$ similarly defined as the top-left and top-right edge lengths. As shown in \Cref{fig: ns bin tree}, the recursive structure of $\mathcal T(\boldsymbol \pi)$ depends on $\rho$: for $\rho = 21$, we have $R_{n+1} = R_n$, $L_{n+1} = L_n + 1 + L_n'$, and $H_{n+1} = \max(L_n + 1 + H_n', H_n)$. In the example shown, $H_{n+1} = H_n = 7 > L_n + 1 + H_n' = 5$. In general, the recursion satisfies:
\begin{align}
    L_{n+1} &= \left\{
    \begin{array}{ll}
        L_n, & \rho = 12\\
        L_n + 1 + L_n', \hspace{3.9pc}& \rho = 21
    \end{array}
    \right. \label{eq: Ln ns rec}\\
    R_{n+1} &= \left\{
    \begin{array}{ll}
        R_n + 1 + R_n', \hspace{3.8pc}& \rho = 12\\
        R_n, & \rho = 21
    \end{array}
    \right.\label{eq: Rn ns rec}\\
    H_{n+1} &= \left\{
    \begin{array}{ll}
        \max(R_n + 1 + H_n',H_n), & \rho = 12\\
        \max(L_n + 1 + H_n', H_n), & \rho = 21
    \end{array}
    \right. \label{eq: Hn ns rec}
\end{align}
These recursions underscore that $(H_n, L_n, R_n)$ no longer obey a simple recurrence, unlike the case for simple butterfly trees (cf.\ \eqref{eq: Hn s rec}).


{
The $\max$ in the height recursions presents the main difficulty in computing the first moment. Nevertheless, the recursive formulas for $L_n$ and $R_n$ permit direct analysis. These formulas parallel the recursion for the number of cycles in a nonsimple butterfly permutation, $C(\boldsymbol\pi)\sim Y_n$. As established in \Cref{thm: old bsbt}, $Y_n$ satisfies a distributional recurrence given by $Y_0 = 1$ and $Y_{n+1} \sim Y_n + \eta_n Y_n'$, where $Y_n'$ is an independent copy of $Y_n$ and $\eta_n\sim\Bern(1/2)$ is independent of $Y_n,Y_n'$. Equations \eqref{eq: Ln ns rec} and \eqref{eq: Rn ns rec} then show that $L_n$ and $R_n$ satisfy the same recurrence up to a shift, with $\mathds 1(\rho = 12) \sim \mathds 1(\rho = 21) \sim \operatorname{Bern}(1/2)$.} We summarize this as:
\begin{lemma}
    Let $\boldsymbol \pi \sim \Unif(\B_n)$. Then $C(\boldsymbol \pi) \sim r(\mathcal T(\boldsymbol \pi))+1 \sim \ell(\mathcal T(\boldsymbol \pi)) + 1$.
\end{lemma}
\begin{remark}
    In both the simple and nonsimple settings, the left-to-right maxima (i.e., the top-right edge length (+1)) of a butterfly tree aligns in distribution with a structural statistic of the underlying permutation: the LIS in the simple case (cf. \Cref{l: simple lis r}), and the number of cycles in the nonsimple case. {The nonsimple case parallels the uniform random permutation scenario, where the left-to-right maxima and the number of cycles both follow a Stirling-1 distribution.}
\end{remark}

%
%
Hence, as shown in \cite{PZ24},  the moments of $L_n$ and $R_n$ can be computed exactly (cf. \cite[Proposition 9]{PZ24}), and a full asymptotic distributional limit has been established (cf. \cite[Theorem 10]{PZ24}). In particular, a direct computation yields
\begin{align}
    \E L_n &= \E R_n = \left(\frac32\right)^n-1\label{eq: 1st moment Ln},\\
    \E L_n^2&= \E R_n^2 = \frac43 \left(\frac32\right)^{2n} - \frac73 \left(\frac32\right)^n + 1 . \label{eq: 2nd moment Ln}
\end{align}
We now have sufficient tools to commence proving \Cref{thm: nsbt}.

Because $L_n \sim R_n$, $H_n \sim H'_n$, and by the recursion in \eqref{eq: Hn ns rec}, we have
\begin{equation}\label{eq: ns recursive mean 1st}
    \E H_{n+1} = \E \max(L_n + 1 + H_n',H_n).
\end{equation}
First, using \eqref{eq: ns recursive mean 1st} and \eqref{eq: 1st moment Ln}, we obtain that
\begin{align}
    \E H_{n+1} \ge \E (L_n + 1 + H_n') = \E(L_n+1) + \E H_n = \left(\frac32\right)^n + \E H_n. \label{eq: lower bd recursion}
\end{align}
This then yields the desired lower bound
\begin{equation}\label{eq: lower bd}
    \E H_n \ge \sum_{k = 0}^{n-1} \left(\frac32\right)^k = 2 \left(\frac32\right)^n - 2.
\end{equation}
\begin{remark}
This lower bound on the mean height of uniform nonsimple butterfly trees matches exactly the mean height for uniform simple butterfly trees (cf. \Cref{thm: sbt}). A similar phenomenon appears for the LIS (cf. \Cref{thm: old sbt,thm: old bsbt}).
\end{remark}

Next,  using  $\max(x,y) = (x + y + |x - y|)/2$ applied to \eqref{eq: ns recursive mean 1st} (along with \eqref{eq: 1st moment Ln}), we can equivalently write
\begin{equation*}
    \E H_{n + 1} = \frac12 \E( L_n + 1 + H_n' + H_n + |L_n + 1 + H_n' - H_n|) = \frac12\left(\frac32\right)^n + \E H_n + \frac12 \E |L_n + 1 + H_n' - H_n|.
\end{equation*}
We can then bound $\E H_n$ by establishing bounds on $\E|L_n + 1 + H_n'-H_n|$. (Note the lower bound \eqref{eq: lower bd recursion} is recovered if applying $\E|Y| \ge |\E Y|$ to this quantity.)

For an upper bound, we have
$$
\E|L_n + 1 + H_n' - H_n| \le  \E |L_n + 1| + \E |H_n - H_n'| \le \left(\frac32\right)^n + \sqrt{2 \operatorname{Var}(H_n)}.
$$
via the triangle inequality, the Cauchy-Schwarz inequality, and \eqref{eq: 1st moment Ln}. Hence,
\begin{equation}\label{eq: upper bd mean}
\E H_{n + 1} \le \left(\frac32\right)^n + \E H_n + \frac12 \sqrt{2\operatorname{Var}(H_n)}
\end{equation}
using \eqref{eq: Hn ns rec}. Next, we will approach a recursive upper bound for $\operatorname{Var}(H_n)$. {Our goal will be to iterate these bounds on $\E H_n$ and $\Var(H_n)$ to derive an upper bound on $\E H_n$ that is sub-linear in $N=2^n$.}

{Now to bound $\Var(H_n)$, we first compute}

\begin{align*}
    \E H_{n + 1}^2 &= \frac12\E \max((L_n + 1 + H_n')^2, H_n^2) + \frac12 \E \max((R_n + 1 + H_n')^2, H_n^2)\\
    &= \E \max((L_n + 1 + H_n')^2, H_n^2) \\
    &= \frac12 \E[ (L_n + 1 + H_n')^2 + H_n^2 + (L_n + 1 + H_n' + H_n)\cdot|L_n + 1 + H_n' - H_n|]\\
    &=\frac12 \E[(L_n + 1)^2 + H_n'^2 + 2H_n' + 2L_nH_n' + H_n^2 + (L_n + 1 + H_n' + H_n)\cdot|L_n + 1 + H_n' - H_n|]\\
    &= \frac12\E[(L_n+1)^2] + \E H_n^2  + (\E L_n + 1)(\E H_n) + \frac12 \E[(L_n + 1 + H_n' + H_n)\cdot|L_n + 1 + H_n' - H_n|].
\end{align*}
{Next, we lower bound}
\begin{align*}
    (\E H_{n+1})^2 &= \left(\frac12 (\E L_n+1) + \E H_n + \frac12 \E|L_n + 1 + H_n' - H_n| \right)^2\\
    &= \frac14(\E L_n + 1)^2 + (\E H_n)^2 + {(\E L_n + 1)(\E H_n)+ \frac14 (\E|L_n + 1 + H_n' - H_n|)^2 } \\
    &\hspace{2pc} + {\frac12(\E L_n + 1)(\E|L_n + 1 + H_n' - H_n|) + \E H_n \cdot  (\E|L_n + 1 + H_n' - H_n|)}\\
    &{\ge \frac14(\E L_n + 1)^2 + (\E H_n)^2 + (\E L_n + 1)(\E H_n)}.
\end{align*}
It follows {that we can upper bound the variance using}
\begin{align*}
    \operatorname{Var}(H_{n+1}) &= \E H_{n+1}^2 - (\E H_{n+1})^2\\
    &{\le \frac12\E[(L_n+1)^2] + \E H_n^2  + (\E L_n + 1)(\E H_n) + \frac12 \E[(L_n + 1 + H_n' + H_n)\cdot|L_n + 1 + H_n' - H_n|]}\\
    &\hspace{2pc}{-\left(\frac14(\E L_n + 1)^2 + (\E H_n)^2 + (\E L_n + 1)(\E H_n)\right)}\\
    &{= \frac12 \E[(L_n + 1)^2] - \frac14 (\E L_n + 1)^2 + \operatorname{Var}(H_n) + \frac12 \E[(L_n + 1 + H_n' + H_n)|L_n + 1 + H_n' - H_n|]}.
\end{align*}

Next, {to control the last term in the above bound,} we see
\begin{align*}
    &\E[(L_n + 1 + H_n' + H_n)\cdot|L_n + 1 + H_n' - H_n|]\\
        &\hspace{2pc} \le \E[(L_n + 1 + H_n' + H_n)(L_n + 1 + |H_n - H_n'|)]\\
        &\hspace{2pc}= \E[(L_n + 1)^2  + H_n'L_n + H_n' + H_n L_n + H_n + H_n'\cdot|H_n - H_n'| + H_n\cdot|H_n - H_n'|]\\
        &\hspace{5pc} + \E[L_n\cdot|H_n - H_n'|] + \E[|H_n - H_n'|]\\
        &\hspace{2pc}=\E[(L_n + 1)^2] + (\E H_n)(\E L_n) + 2 \E H_n + \E[H_n L_n] + 2\E[H_n\cdot|H_n - H_n'|] \\
        &\hspace{5pc}+ \E[L_n\cdot|H_n - H_n'|] + \E [|H_n - H_n'|]\\
        &{\hspace{2pc} \le \E[(L_n + 1)^2] + (\E H_n)(\E L_n + 2) + \sqrt{\E H_n^2 \cdot \E L_n^2}} \\
        &\hspace{5pc}{+ 2 \sqrt{\E H_n^2\cdot 2\operatorname{Var}(H_n)}+ \sqrt{\E L_n^2 \cdot 2\operatorname{Var}(H_n)} + \sqrt{2 \operatorname{Var}(H_n)}}\\
        &\hspace{2pc} \le \E[(L_n + 1)^2] + (\E H_n)(\E L_n + 2) + \sqrt{(\operatorname{Var}(H_n) + (\E H_n)^2)\cdot \E L_n^2} \\
        &\hspace{5pc}+ 2 \sqrt{2(\operatorname{Var}(H_n) + (\E H_n)^2)\operatorname{Var}(H_n)}+ \sqrt{2\E L_n^2 \operatorname{Var}(H_n)} + \sqrt{2 \operatorname{Var}(H_n)}
\end{align*}
using the triangle inequality and the Cauchy-Schwarz inequality. {Combining the previous inequalities yields}
\begin{align*}
    \operatorname{Var}(H_{n+1}) 
    &{\le \operatorname{Var}(H_n) + {\E[(L_n + 1)^2] - \frac14(\E L_n + 1)^2} + \frac12(\E H_n)(\E L_n + 2) }\\
    &\hspace{2pc} + \frac12 \sqrt{(\operatorname{Var}(H_n) + (\E H_n)^2)(\E L_n^2)} + \sqrt{2(\operatorname{Var}(H_n) + (\E H_n)^2)\operatorname{Var}(H_n)}\\
    &\hspace{2pc} + \frac12\sqrt{2\E L_n^2 \operatorname{Var}(H_n)} + \frac12 \sqrt{2 \operatorname{Var}(H_n)}.
\end{align*}
Let $\lambda = 3/2$. Note that  $\E L_n^2  = \frac43 \lambda^{2n} - (\frac73 \lambda^n - 1) \le \frac43 \lambda^{2n}$ and $\E L_n + 1= \lambda^n$ by \eqref{eq: 1st moment Ln} and \eqref{eq: 2nd moment Ln}. It follows that
\begin{align*}
    \operatorname{Var}(H_{n+1}) &\le \operatorname{Var}(H_n) + {\left[\frac43 \lambda^{2n} - \frac73 \lambda^n + 1 + 2 \lambda^n - 1\right] - \frac14 \lambda^{2n}}+ \frac12(\E H_n)(\lambda^n + 1) \\
    &\hspace{2pc} + \frac12 \sqrt{\frac43\lambda^{2n}} \sqrt{\operatorname{Var}(H_n) + (\E H_n)^2} + \sqrt{2(\operatorname{Var}(H_n) + (\E H_n)^2)\operatorname{Var}(H_n)}\\
    &\hspace{2pc} + \frac12\sqrt{\frac83 \lambda^{2n}} \sqrt{\operatorname{Var}(H_n)} + \frac12 \sqrt{2 \operatorname{Var}(H_n)}\\
    &\le \operatorname{Var}(H_n) + {\frac{13}{12} \lambda^{2n} +} \frac12(\E H_n)(\lambda^n + 1) \\
    &\hspace{2pc} + \sqrt{\frac13} \lambda^n \sqrt{\operatorname{Var}(H_n) + (\E H_n)^2} + \sqrt 2 \cdot \sqrt{(\operatorname{Var}(H_n) + (\E H_n)^2)\operatorname{Var}(H_n)}\\
    &\hspace{2pc} + \sqrt{\frac23} \lambda^n \sqrt{\operatorname{Var}(H_n)} + \frac12 \sqrt{2 \operatorname{Var}(H_n)}.
\end{align*}
Together with \eqref{eq: upper bd mean}, these establish recursive bounds relating $\E H_n$ and $\operatorname{Var}(H_n)$ to each other, which we will now utilize to find a better upper bound on $\E H_n$.

Let $a_n, b_n\ge 0$ with $a_0 = b_0 = 0$ such that
\begin{align}
    a_{n+1} &= \lambda^n + a_n + \frac12 \sqrt{ 2b_n},\label{recurrencea}\\
    b_{n+1} &= b_n + {\frac{13}{12} \lambda^{2n} +} \frac12 (\lambda^n + 1)a_n + \sqrt{\frac13} \lambda^n\sqrt{b_n + a_n^2} \label{recurrencea}\nonumber\\
    &\hspace{2pc}+ \sqrt{(b_n + a_n^2)2b_n} + \left(\sqrt{\frac13} \lambda^n + {\frac12}\right)\sqrt{2b_n}.\nonumber
\end{align}
Our goal is to bound both $\E H_n$ and $\operatorname{Var}(H_n)$ in terms of $a_n,b_n$. 
By construction and induction, we have 
$$\E H_n \le a_n \  \mbox{and} \ \operatorname{Var}(H_n) \le b_n \qquad  \mbox{for all } n.$$ 
Next, observe that
\begin{equation}\label{eq:a^2}
a_{n+1}^2 = \lambda^{2n} + a_n^2 + \frac12 b_n + 2\lambda^n a_n + (\lambda^n + a_n)\sqrt{2 b_n}.
\end{equation}
We aim to find a constant $C \ge 0$ such that $b_n \le C a_n^2$. The strategy is to preserve this inequality inductively through the recursion. The recurrence \eqref{recurrencea} for $a_{n+1}$ depends only on $a_n$ and $\sqrt{b_n}$. Thus, once we establish the quadratic control $b_n \le C a_n^2$, we obtain $\sqrt{b_n}\le \sqrt{C}\,a_n$, and the coupled two-variable system reduces to a one-dimensional linear recurrence involving only $a_n$. This yields an inequality of the form 
\begin{equation*}
    a_{n+1} \le \lambda^n + \xi a_n,
\end{equation*}
where $\xi = 1 + \tfrac{\sqrt{2C}}{2}$. The standard solution to this
linear recurrence inequality gives
\[
a_n \le A \lambda^n + B \xi^n
\]
for suitable constants $A,B$. Moreover, if $\frac12 < C < 2$ then $\lambda < \xi < 2$, which would yield an upper bound on $a_n$ that is sub-linear {in $N=2^n$}, with exponent $\log_2 \xi < 1$. As $\E H_n \le a_n$, the same upper bound applies to $\E H_n$.

The remainder of this section will verify that this bound holds with the choice 
\begin{equation}\label{eq: C*}
    C^* := 1 + \sqrt{8\sqrt 2 - 11} \approx 1.5601.
\end{equation}
This in turn yields the desired sublinear upper bound on $\E h(\mathcal T(\pi_n))$ for $\pi_n \sim \Unif(\B_n)$, as stated in \Cref{thm: nsbt}, with exponent $\beta = \log_2 \xi \approx 0.913189$. 

\begin{lemma}\label{l: bn<Can^2}
    Let $C \ge C^*$. Then $b_n \le C a_n^2$ for all $n \ge 0$.
\end{lemma}
\begin{proof}
    The result trivially holds for any $C$ when $n=0$ {since $a_0 = b_0 = 0$}. Now assume the result holds for $n\ge0$.  Using the inequality $\sqrt{b_n + a_n^2} \le \sqrt{b_n} + a_n$, we obtain that
\begin{align*}
    b_{n+1} &\le (1 + \sqrt 2) b_n + \frac{13}{12} \lambda^{2n} + \left(\left(\frac12 + \sqrt{\frac13}\right) \lambda^n +  {\frac12}\right) a_n  + \left(\left(\sqrt{\frac16} + \sqrt{\frac13}\right) \lambda ^n  + a_n +   {\frac12}\right) \sqrt{2b_n} .
\end{align*}
To get a potential value of $C$, we next compute
\begin{align}
    b_{n+1} - Ca_{n+1}^2 &\le -2\left(\left(C - \frac12\left(\frac12 + \sqrt{\frac13}\right)\right) \lambda^n - {\frac14} \right) a_n - \left(\left(C - \left(\sqrt{\frac16} + \sqrt{\frac13}\right)\right) \lambda^n - \frac12 \right) \sqrt{2 b_n}\nonumber\\
    &\hspace{2pc} - \left(C-\frac{13}{12}\right) \lambda^{2n}+\left[\left(1 + \sqrt 2 - \frac{C}2\right) b_n-Ca_n^2 -(C-1)a_n \sqrt{2 b_n}\right] \nonumber\\    &=:\fbox{I}+\fbox{II}+\fbox{III}+\fbox{IV}.\label{eq:ind upper bound}
\end{align}
using also \eqref{eq:a^2}. One way to ensure that $b_{n+1} \le C a_{n+1}^2$ is to choose $C$ such that the upper bound has non-positive coefficients for each of the three terms involving $a_n$, $\sqrt{2 b_n}$, and $\lambda^{2n}$ (i.e., \fbox{I} through \fbox{III}) and that the term inside the square brackets \fbox{IV} is also non-positive. This will be the approach we take. {We will establish that $C \ge C^*$ is a sufficient condition to ensure that each of these four terms is non-positive.} Future refinements could allow some terms to be positive, as long as their overall sum remains non-positive.  

In order for all of the three terms \fbox{I}, \fbox{II}, and \fbox{III} to be non-positive, the constant $C$ must satisfy 
\begin{equation*}
    C\geq \max\left\{\frac12\left(\frac12 + \sqrt{\frac13}\right)+\frac{1}{4\lambda^n},\sqrt{\frac16} + \sqrt{\frac13} + \frac{1}{2\lambda^n},\frac{13}{12}\right\}.
\end{equation*}
{Since $\lambda^n \ge 1$}, it suffices for $C$ to satisfy
\begin{equation*}
    C\geq \max\left\{\frac12\left(\frac12 + \sqrt{\frac13}\right)+\frac{1}{4},\sqrt{\frac16} + \sqrt{\frac13} + \frac{1}{2},\frac{13}{12}\right\}.
\end{equation*}
Since $\frac12\big(\frac12 + \sqrt{\frac13}\big)+\frac14 \approx 0.7887$, $\sqrt{\frac16} + \sqrt{\frac13} + \frac12 \approx 1.48560$, and $\frac{13}{12} \approx 1.08333$, {choosing $C\ge C^* \approx 1.5601$} {would be sufficient} to ensure that the three terms \fbox{I}, \fbox{II}, and \fbox{III} are all non-positive.

To further refine our choice of $C$, we can find a sufficient lower bound to ensure the remaining term \fbox{IV} is also non-positive. Using the inductive hypothesis $b_n \le C a_n^2$, we have 
\begin{equation*}
    b_n \le \sqrt{C} a_n \sqrt{b_n} \le C a_n^2.
\end{equation*}
Applying this to each term in \fbox{IV} then establishes
\begin{equation*}
    \fbox{IV} = \left(1 + \sqrt 2 - \frac{C}2\right) b_n-Ca_n^2 -(C-1)a_n \sqrt{2 b_n} \le f(C) \cdot  b_n
\end{equation*}
for $$f(x) = \sqrt 2 - \frac{x}2 - \sqrt {\frac 2x} \cdot (x-1).$$ We now show $f(x) \le 0$ iff $x \ge C^*$ for $C^*\approx 1.5601$ from \eqref{eq: C*}. {Since $f'(x)<0$ for $x>0$, with $f(x)\to +\infty$ as $x\to 0^+$ and $f(x)\to -\infty$ as $x\to+\infty$, the function has a unique zero on $(0,\infty)$. Setting $y=\sqrt{x}$, one checks that $f(x)=0$ is equivalent to $y f(y^2)=0$, a cubic equation in $y$ with exactly one positive root (by Descartes' rule of signs), which we denote by $y_*$; this root can be found using Cardano's formula, and the corresponding zero of $f$ is $C^*=y_*^2$. We omit the algebraic details.} Hence, \fbox{IV} is non-positive if $C \ge C^*$.  This choice also ensures \fbox{I} through \fbox{III} are each non-positive. Applying this to \eqref{eq:ind upper bound}  yields $b_{n+1} \le C a_{n+1}^2$. This completes the inductive step.
\end{proof}

\Cref{l: bn<Can^2} now establishes $b_n \le C a_n^2$ for any $C \ge C^*$. In particular, this holds for $C = C^*$, which we will now assume to be the case going forward. Now we have
\begin{equation*}
    a_{n+1} = \lambda^n + a_n + \frac12 \sqrt{2 b_n} \le \lambda^n + \xi a_n, \qquad \xi=1 + \frac{\sqrt {2C^*}}2 =\frac12\left(1 + \sqrt 2 + \sqrt{2 \sqrt 2 - 1}\right)
    \approx 1.88320.
\end{equation*}
Solving the upper bound recursion $a_{n+1} \le \lambda^n + \xi a_n$ with $a_0 = 0$ yields the bound
\begin{equation}\label{eq:upper bound}
    \E H_n \le a_n \le  \frac{\xi^n-\lambda^n}{\xi - \lambda} .
\end{equation}
Hence, noting $1.88320 \approx \xi>\lambda=1.5$, we thus have the desired sub-linear upper bound
\begin{equation*}
    \E H_n \le d N^\beta \cdot (1 + o(1)), \qquad \beta = \log_2\xi \approx 0.913189, \qquad \quad 
    d = \frac1{\xi-\lambda}=\frac{2}{\sqrt{2C^*}-1} 
    \approx 2.60958.
\end{equation*}
{Combining the above upper bound and the lower bound in \eqref{eq: lower bd}} completes the proof of \Cref{thm: nsbt}.

\begin{figure}[t]
    \centering
    \includegraphics[width=0.7\linewidth]{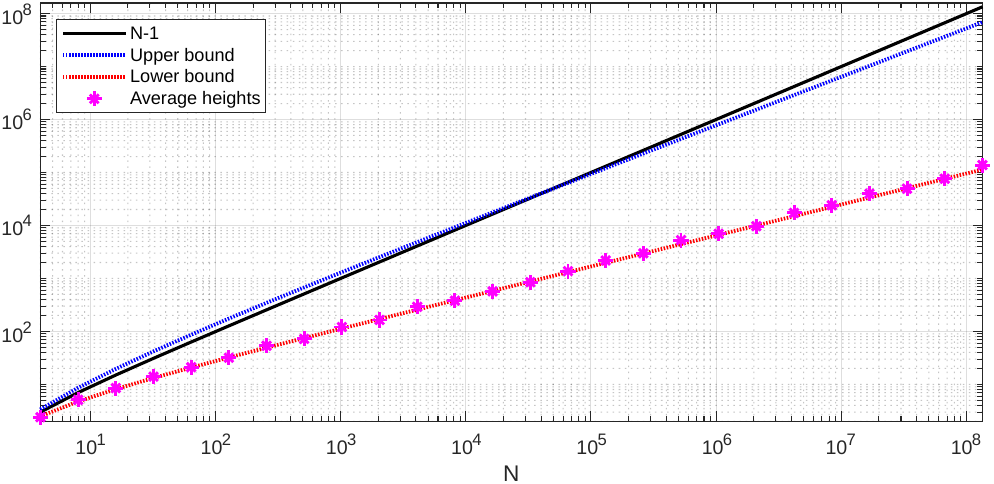}
    \caption{Comparison of average heights of  sampled nonsimple butterfly trees (10 trials per size) to the lower and upper bounds from \Cref{thm: nsbt}}
    \label{fig:lower_upper}
\end{figure}

\paragraph{Discussion.} 
    While the upper bound in \Cref{thm: nsbt} is sub-linear, this is smaller than $N-1$, the maximal possible height of a BST with $N = 2^n$ nodes, starting only at $n=16$ where $N=65{,}536$ (when there are $2^{65{,}535} \approx 1.002 \cdot 10^{19{,}728}$ nonsimple butterfly trees). \Cref{fig:lower_upper} shows a comparison of the bounds \eqref{eq: lower bd} and \eqref{eq:upper bound} from \Cref{thm: nsbt} to $N-1$.

    We compute the \textit{exact} expected heights of butterfly trees for small values of $n \leq 4$. For $n = 1, 2, 3$, the expected heights of simple and nonsimple butterfly trees agree, with respective values $1$, $2.5$, and $4.75$. At $n = 4$, however, the expected height of nonsimple butterfly trees ($8.208984375$) exceeds that of the simple ones ($8.125$). This discrepancy arises because the recursive form of \( H_n \) in \eqref{eq: Hn ns rec} can now attain the height of the internal root subtree: the event \( L_n + 1 + H_n' > H_n \) can first occur when \( n \ge 3 \). This behavior is clearly visible in \Cref{fig: ns bin tree}, where the height at level 4 begins to reflect the contribution of the top internal butterfly tree. We do not pursue exact enumeration for $n = 5$, as there are $2^{31} = 2{,}147{,}483{,}648$ such trees with $N = 2^5 = 32$ nodes. Using \eqref{eq: lower bd recursion}, it follows that this lower bound—given by the exact mean for simple butterfly trees—remains a \textit{strict} inequality for $n \ge 4$. Still, these results suggest that the lower bound  stays close to the true mean even as $n$ increases.

    Empirical evidence supports this observation. \Cref{fig:lower_upper} compares the average heights of 10 iid uniform nonsimple butterfly trees with $N = 2^n$ nodes for $n \le 27$ (up to $N = 2^{27} = 134{,}217{,}728$ nodes, that then corresponds to $2^{2^{27}-1} \approx 5.98 \cdot 10^{40{,}403{,}561}$ such trees), plotted alongside the theoretical bounds from \Cref{thm: nsbt}. A custom implementation using the recursions \eqref{eq: Ln ns rec}, \eqref{eq: Rn ns rec}, and \eqref{eq: Hn ns rec} enables efficient direct sampling of  butterfly tree heights. A least-squares fit of the form $\E [h_n^{\B}] \approx \eta N^{{\psi}}$ applied to the sampled data yields the estimates $\hat \eta = 1.593895$ and ${\hat \psi} = 0.606059$, with strong fit quality ($R^2 = 0.998577$). While the lower bound is strict for $n\ge 4$, the empirical fit suggests that its scaling may nevertheless be asymptotically {tight}, i.e., ${\psi} = \alpha \approx 0.58496$.

    Further supporting this trend, we compute the heights of $100{,}000$ nonsimple butterfly trees with $N = 2^{10} = 1{,}024$ nodes—each sampled from an iid uniform nonsimple butterfly permutation (of which there are $2^{1{,}023} \approx 8.988 \cdot 10^{307}$). The resulting sample mean of $118.66$ lies much closer to the lower bound of $113.33$ than to the upper bound of $1{,}313.53$ (which again exceeds the maximal possible height of $N - 1 = 1{,}023$). Notably, the sample minimum of $56$ lies below the minimal height of $62$ for simple butterfly trees (cf.\ \Cref{t: s height cnts}), while the sample maximum of $263$ falls below that of 2\% ($22/1{,}024$) of the simple trees. \Cref{fig:sample bsts histogram} displays a histogram of the sampled heights {scaled by $N^\alpha=(3/2)^{10} \approx 57.6650$, yielding a scaled sample mean of $2.0588$}.  {The histogram shape also suggests a potential limiting distribution for $H_n / N^\alpha$ that parallels the theta limiting distribution family arising in the height analysis of random rooted planar trees (e.g., see Proposition~V.4 in~\cite{flajolet_sedgewick}) and (unlabeled) binary trees \cite{flajolet1982average}.}


\begin{figure}
    \centering
    \includegraphics[width=0.7\linewidth]{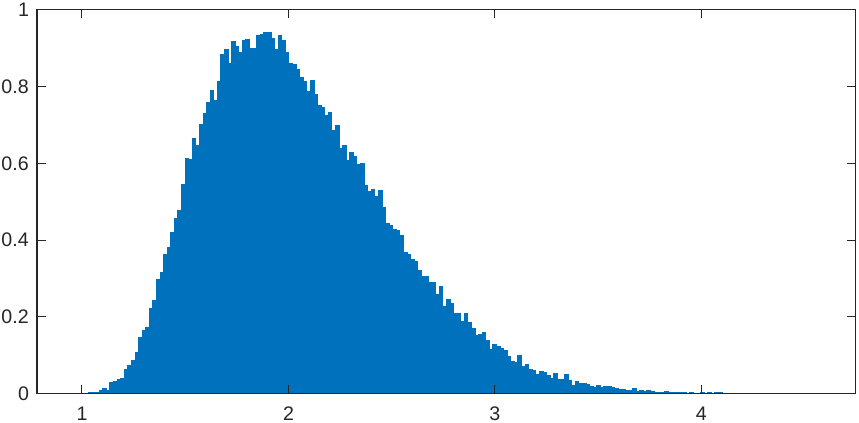}
    \caption{{Scaled h}eights for nonsimple butterfly trees with $N = 2^{10} = 1{,}024$ nodes, 100{,}000 trials}
    \label{fig:sample bsts histogram}
\end{figure}



\bibliographystyle{plain} 
\bibliography{references}
\end{document}